\documentclass[]{article}  
\usepackage[width=14.0cm, height=21.0cm]{geometry}

\usepackage{amsmath, amsfonts, amsthm, amssymb, xypic, color, epigraph, mathrsfs, calc}
\usepackage[hidelinks]{hyperref}

\usepackage{tensor,dsfont, enumerate, enumitem}
\usepackage{tikz}
\usetikzlibrary{arrows, automata}

\usepackage{graphicx}
\usepackage{float}
\usepackage{verbatim}
\usepackage{makeidx}
\usepackage{longtable}

\numberwithin{equation}{section}
\allowdisplaybreaks

\newcommand{\ie}{\emph{i.e.\ }}
\newcommand{\eg}{\emph{e.g.\ }}

\newtheorem{theo}{Theorem}[section]
\newtheorem{lma}[theo]{Lemma}
\newtheorem{cor}[theo]{Corollary}
\newtheorem{defn}[theo]{Definition}
\newtheorem{cond}[theo]{Condition}
\newtheorem{prop}[theo]{Proposition}
\newtheorem{conj}[theo]{Conjecture}

\DeclareMathOperator{\R}{\mathbb{R}}

\DeclareMathOperator{\N}{\mathbb{N}}
\DeclareMathOperator{\Z}{\mathbb{Z}}
\DeclareMathOperator{\cN}{\mathcal{N}}

\DeclareMathOperator{\E}{\mathbb{E}}

\DeclareMathOperator{\Haus}{\mathscr{H}}
\DeclareMathOperator{\Prob}{\mathbb{P}}

\DeclareMathOperator{\fS}{\mathfrak{S}}
\DeclareMathOperator{\Var}{Var}
\DeclareMathOperator{\ess}{ess}

\renewcommand{\epsilon}{\varepsilon}

\newcommand{\blfootnote}[1]{
\begingroup
\renewcommand\thefootnote{}\footnote{#1}
\addtocounter{footnote}{-1}
\endgroup
}

\author{Sascha Troscheit}

\date{
	\emph{\small Department of Pure Mathematics, University of Waterloo, Waterloo, Ont., N2L 3G1, Canada}\\[0.5ex] 
	\small \href{mailto:stroscheit@uwaterloo.ca}{\texttt{stroscheit@uwaterloo.ca}}\\[2ex]
	\today}
\title{Exact Hausdorff and packing measures for random self-similar code-trees with necks}

\begin{document}
\vspace{-10cm}
\maketitle

\begin{abstract}
	Random code-trees with necks were introduced recently to generalise the
	notion of $V$-variable and random homogeneous sets. While it is known that the
	Hausdorff and packing dimensions coincide irrespective of overlaps, their exact Hausdorff
	and packing measure has so far
	been largely ignored. In this article we consider the general question
	of an appropriate gauge function for positive and finite Hausdorff and packing
	measure.
	We first survey the current state of knowledge and establish some bounds on these gauge
	functions. We then show that self-similar code-trees do not admit a gauge functions that simultaneously
	give positive and finite Hausdorff measure almost surely.
	This surprising result is in stark contrast to the random recursive model and sheds some
	light on the question of whether $V$-variable sets interpolate between
	random homogeneous and random recursive sets. We conclude by discussing implications of our
	results.
\end{abstract}

\tableofcontents

\clearpage
The Hausdorff dimension and measure of random constructions such as branching processes, Brownian
motion and stochastically self-similar sets has been
studied since the 1980s and much progress has been made on conditions for
positive and finite Hausdorff and packing measure. We refer the reader to the seminal work of
Athreya~\cite{AthreyaBook} on Branching processes and Watanabe~\cite{Watanabe07}, Liu~\cite{Liu00}
and \cite{Liu96} for recent progress on the Hausdorff and packing measure of Galton--Watson
processes. The related stochastically self-similar sets were analysed by Graf, Williams, and
Mauldin \cite{Graf88}, \cite{Mauldin87}; Berlinkov and Mauldin \cite{Berlinkov02}; and Berlinkov
\cite{Berlinkov03}; and we will come back to those in Section~\ref{sect:survey}. Apart from these 
processes, the question of exact Hausdorff and packing measures was answered for some random
re-orderings by Hu~\cite{Hu95} and \cite{Hu96}, and for self-avoiding walks on the Sierpi\'nski
Gasket by Hattori \cite{Hattori00}. For deterministic sets, Olsen~\cite{Olsen03a} considered the
exact Hausdorff measure on some Cantor sets.

Despite this great canon of work, random homogeneous
attractors have largely been ignored, even though they represent a very
natural model for random sets. 
In this article we remedy this gap by giving bounds on the exact gauge functions required and
showing that there is no gauge function which simultaneously gives positive and finite
Hausdorff measure; a stark contrast to all other examples just mentioned.
We start by defining random code-trees in a similar spirit to the seminal
papers by J\"arvenp\"a\"a et al.~\cite{Jarvenpaa14a, Jarvenpaa16, Jarvenpaa17}
in Section~\ref{sect:Intro}, and provide historical context to our results in
Section~\ref{sect:survey}. 

Since the model of random code trees is fairly abstract we continue in
Section~\ref{sect:boundshomogeneous} by reducing the model to random
homogeneous attractors with equal contraction ratios. These
\emph{equicontractive homogeneous attractors} are simpler to study and we first
state gauge function that give a finer quantification of the Hausdorff
dimension and then prove that there cannot be a gauge function that gives
positive and finite Hausdorff measure almost surely.  
This is followed by a statement and proof of the general theorem for
self-similar random code-trees with necks in Section~\ref{sect:generalcase}.

We end this article by proving some analogous results for the packing measure
in Section~\ref{sect:packingquestion} and discuss wider implications of our
result in Section~\ref{sect:implications}.

\section{Introduction and Random Models}\label{sect:Intro}
Let $k\in\N$, and let $\Lambda\subset\R^k$ be a non-empty compact set. We will
use $\Lambda$ to index our random choice of iterated function systems 
 and associate with it a Borel probability measure $\mu$ compactly supported on
$\Lambda$. 
For $\lambda\in\Lambda$ let $\mathbb{I}_{\lambda}=\{f_\lambda^1, f_\lambda^2,
  \dots, f_\lambda^{\cN_\lambda}\}$ be a collection of $\cN_\lambda\in\N$
  contracting similarities on $\R^d$, \ie maps that satisfy $\lvert f_\lambda^i
  (x) - f_{\lambda}^i(y)\rvert = c_\lambda^i\lvert x-y\rvert$ for some
  $c_\lambda^i \in (0,1)$.
Finally, let $\mathbb{L}=\{\mathbb{I}_{i}\}_{i\in\Lambda}$ be a (not
necessarily finite) collection of iterated function systems with at most $\cN$
similarities. We will refer to the pair $(\mathbb{L},\mu)$ as a \emph{random
iterated function system (RIFS)}. Unless otherwise noted we assume 
\[
2\leq\cN:=\sup_{\lambda\in\Lambda}\cN_\lambda < \infty\text{ and }
0<c_{\min}:=\inf_{\lambda\in\Lambda}\,\min_{1\leq i\leq \cN_\lambda}
c_\lambda^i \leq  \sup_{\lambda\in\Lambda}\,\max_{1\leq i\leq \cN_\lambda}
c_\lambda^i =:{c}_{\max}< 1.
\]

\subsection{Random code-trees and their attractors}
\subsubsection{The general model}\label{sect:code-trees}
Consider the rooted $\cN$-ary tree. The general idea of random code-trees is
achieved by `randomly' picking a labelling function $\tau$ that labels each
node with a single $\lambda\in\Lambda$, chosen according to some probability
measure $\Prob$. We first describe the general set-up, before talking about
specific methods of picking the function $\tau$.
 
We denote the space of all possible functions (and hence labellings) of the
full tree by $\mathcal{T}$ and refer to individual realisations by
$\tau\in\mathcal{T}$.
In this full tree we address vertices by which branch was taken; if $v$ is a
node at level $k$ we write $v=(v_{1},v_{2},\dots,v_{k})$, with
$v_{i}\in\{1,\dots,\cN\}$ and root node $v=(.)$.
We write $\Sigma_k$ for the nodes at level $k$ and
$\Sigma^*=\bigcup_{k\in\N_0}\Sigma_k$ for the set of all nodes, where
$\N_0=\N\cup\{0\}$ and $\Sigma_0=\{(.)\}$. Thus,
\[
\Sigma^*=\{\{(.)\}, \{(1),(2),\dots,(\cN)\}, \{ (1,1),(1,2),\dots,(1,\cN),(2,1),\dots,(\cN,\cN)\},\dots\;\}.
\]
We slightly abuse notation and consider $\tau$ both as a function
$\tau:\Sigma^*\to\Lambda$, where $\tau(v)\in\Lambda$ and as a labelled full
tree. Given a node $v\in\Sigma^*$ we define the shift $\sigma^{v}\tau$ to be
the full subtree starting at vertex $v$, with $\sigma^{(.)}\tau=\tau$. 
At this point we note that since $\Lambda$ was a compact topological space, the set of all
realisations $\mathcal{T}$ is also a compact topological space  with respect to the obvious product
topology by Tychonoff's theorem.

We write $e_{\lambda}^{j}$ for the letter representing the map
$f_{\lambda}^{j}\in\mathbb{I}_{\lambda}$.
For each labelled full tree $\tau$, we construct another rooted labelled $\cN$-ary tree
$\mathbf{T}_{\tau}$, where each node is labelled by a `coding' describing a composition of maps.
Given two codings $e_1$ and $e_2$, we write $e_1 e_2 = e_1 \odot e_2$ for concatenation. 
We let $\epsilon_0$ be the empty word and use the symbol $\emptyset$ as a multiplicative zero, \ie
$\emptyset\odot e = e\odot\emptyset=\emptyset$, to represent the empty function. Similarly, if
$\{e_1,\dots ,e_n\}$ is a collection of codings, then $\{e_1,\dots ,e_n\}\cup\emptyset=\{e_1,\dots,
e_n\}$. This letter $\emptyset$ is used to `delete' a subbranch if the number of maps in an IFS is
less than $\cN$.

\begin{defn}\label{infinityDef1}
	Let $\mathbf{T}_{\tau}$ be a labelled tree, 
	we write $\mathbf{T}_{\tau}(v)$ for the label of node $v$ of the tree $\mathbf{T}_{\tau}$.
	The \emph{code-tree} $\mathbf{T}_{\tau}$ is then defined inductively:
	\[
	\mathbf{T}_{\tau}((.))=\epsilon_{0}\text{ and }
	\mathbf{T}_{\tau}(v)=\mathbf{T}_{\tau}((v_{1},\dots,v_{k}))=\mathbf{T}_{\tau}((v_{1},\dots,v_{k-1}))\odot
	e_{\tau(v_{k-1})}^{v_{k}}
	\]
	for $1\leq v_{k}\leq \cN_{\tau(v_{k-1})}$ and $e_{\tau(v_{k-1})}^{v_{k}}=\emptyset$
	otherwise. 	
	We refer to the the set of all codings at the $k$-th level by 
	\[
	\mathbf{T}_{\tau}^{k}=\bigcup_{v\in\Sigma_k}\mathbf{T}_{\tau}(v).
	\]
\end{defn}

We can now define the attractor of the code-tree.
\begin{defn}\label{infinityDef2}
	Let $\mathbb{L}$ be a collection of IFS and let $\tau\in\mathcal{T}$. The \emph{attractor of a code-tree} $F_{\tau}$ is the compact set satisfying
	\[
	F_{\tau}=\bigcap_{k=1}^{\infty}\bigcup_{e\in\mathbf{T}_{\tau}^{k}} f_{e_1}\circ f_{e_2} \circ \dots \circ f_{e_k}(\Delta),
	\]
	where $\Delta$ is a sufficiently large compact set, satisfying $f_{\lambda}^i(\Delta)\subseteq\Delta$ for all $\lambda\in\Lambda$ and $1\leq i\leq \cN_\lambda$.
\end{defn}

This general -- and somewhat abstract -- way of describing geometric objects is very
flexible.
For example, let $\Lambda=\{0,1,2,3\}$, $f_l(x)=x/2$, $f_r(x)=x/2+1/2$ and set $\mathbb{I}_0=\left\{
\right\}$, $\mathbb{I}_1=\left\{ f_l \right\}$, $\mathbb{I}_2=\left\{ f_r \right\}$, and
$\mathbb{I}_3=\left\{ f_l , f_r \right\}$.
Then, $\mathbf{T}_\tau^k$ represents subsets of all dyadic intervals of length $2^{-k}$ and $F_\tau$
is the $\limsup$ set of a sequence of decreasing dyadic intervals. Therefore, constructing $\tau$ in
the appropriate way, we can recover every compact subset of the unit interval with a code-tree. 

Instead of constructing code-trees with a certain set in mind, we could also choose $\tau$ at random. 
In the above example, choosing each $\mathbb{I}_i$ with probability $1/4$ at every step in the
construction gives rise to Mandelbrot percolation of the unit line, an example of stochastically
self-similarity.
We now describe the main ways of choosing $\tau$.

\subsubsection{The random recursive measure} Random recursive attractors are random sets that
exhibit a stochastic self-similarity. They were first investigated in the 1980s by
Falconer~\cite{Falconer86} and Graf~\cite{Graf87} and we will summarise their and later results in
Section~\ref{sect:survey}.  These random fractals satisfy the following equality (in distribution),
where $\lambda$ is chosen according to some compactly supported  Borel probability $\mu$ on
$\Lambda$.  \[ F_\tau =_{d} \bigcup_{i=1}^{\cN_\lambda}f_\lambda^i(F_{\tau}) \] There exists a
natural measure $\Prob_T$ on the collection of code-trees, induced by $\mu$ which describes the same
model. We avoid giving a description here, and briefly comment that $\Prob_T$ can be
obtained by choosing $\tau$ such that for every open set $\mathcal{O}\subseteq\Lambda$, the
probability that $\tau(v)\in\mathcal{O}$ is $\mu(\mathcal{O})$ for every $v\in \Sigma^*$. Further,
given distinct $v,w\in\Sigma^*$ and (not necessarily distinct) open sets
$\mathcal{O}_v,\mathcal{O}_w\subseteq\Lambda$, the probability that $\tau(v)\in\mathcal{O}_v$ and
$\tau(w)\in\mathcal{O}_w$ are independent.

\subsubsection{The homogeneous measure} Another natural measure, $\Prob_H$, is obtained by choosing
an iterated function system of $\mathbb{L}$ according to $\mu$ at every level $k$ of the
construction and applying the same random IFS to \emph{all} nodes at level $k$. While the IFS 
is still chosen i.i.d.\ with respect to the tree levels, all nodes at the same level share the same label.
This is why it is called the \emph{homogeneous} measure.

For this model one does not need the full abstract model of random code-trees and we will use the
following, somewhat simpler, notation.   Consider each IFS
$\mathbb{I}_\lambda$ as a self-map on compact subsets of $\R^d$. That is $\mathbb{I}_\lambda:
\mathcal{K}(\R^d)\to\mathcal{K}(\R^d)$, given by $\mathbb{I}_\lambda(K)=
\bigcup_{f\in\mathbb{I}_\lambda}f(K)$. We will index random realisations by an infinite sequence
with entries in $\Lambda$. The \emph{set of realisations}, denoted by $\Omega$, is given by
$\Omega=\Lambda^{\N}$ and realisations $\omega\in\Omega$ are chosen according to the product
(probability) measure $\Prob_H=\mu^{\N}$.

\begin{defn}\label{defn:1varcoding}\index{coding}
  The \emph{$k$-level coding} with respect to realisation
  $\omega=\omega_1 \omega_2\dots\in\Omega=\Lambda^{\N}$ is 
  \[
    \mathbf{C}^{k}_{\omega}=\bigcup_{1\leq j_i \leq\cN_{\omega_i}}
  e_{\omega_1}^{j_1}e_{\omega_2}^{j_2}\dots e_{\omega_k}^{j_k} \;\;(k\in\N)\;\;\text{ and }\;\;
\mathbf{C}^{0}_{\omega}=\epsilon_{0}. 
\]
The \emph{set of all finite codings}
$\mathbf{C}_{\omega}^{*}$ is defined  by \[ \mathbf{C}_{\omega}^{*}=\bigcup_{i=0}^{\infty}
  \mathbf{C}_{\omega}^{i}.  \] 
\end{defn}

\begin{defn}\label{defn:levelcoding} The \emph{$k$-level prefractal $F_{\omega}^{k}$} and the
  \emph{random homogeneous random attractor} $F_{\omega}$ are \[
    F_{\omega}^{k}=\mathbb{I}_{\omega_1}\circ \mathbb{I}_{\omega_2}\circ \dots \circ
    \mathbb{I}_{\omega_k}(\Delta)=\bigcup_{e\in \mathbf{C}_{\omega}^k} f_{e_1}\circ f_{e_2} \circ
    \dots \circ f_{e_k}(\Delta) \] and \[ F_{\omega}=\bigcap_{k=1}^{\infty}F_\omega
    ^k=\bigcap_{k=1}^{\infty}\bigcup_{e\in \mathbf{C}_{\omega}^k} f_{e_1}\circ f_{e_2} \circ \dots
  \circ f_{e_k}(\Delta), \] where $\Delta\in\mathcal{K}(\R^d)$ is such that
  $f_{\lambda}^i(\Delta)\subseteq\Delta$ for all $\lambda\in\Lambda$ and $1\leq i \leq\cN_\lambda$.
\end{defn}

\subsubsection{$V$-variable sets and random code-trees with necks}

$V$-variable sets were first introduced by Barnsley et al.~\cite{Barnsley05,Barnsley08,Barnsley12}
and are characterised by allowing up to $V\in\N$ different structures at every level of the
construction, see also Freiberg~\cite{Freiberg10} for a recent survey.  A more general model was developed by J\"{a}rvenp\"{a}\"{a} et
al.~\cite{Jarvenpaa14a,Jarvenpaa16,Jarvenpaa17} in the context of self-affine maps with random
translates and is the model that we will adopt in this manuscript.  We note that setting this model
up in the right way allows us to recover both $V$-variable and random homogeneous attractors.

The central property that was crucial for the proofs in both the $V$-variable and the code-tree
setting was the almost sure existence of \emph{necks}. Informally, these necks are levels in the
construction at which point all subtrees are identical. Thus, these models still possess some
homogeneity which is exploited in proofs.

\begin{defn} Let $\mathcal{T}$ be the space of all mappings $\tau:\Sigma^*\to\Lambda$.  Let
  $\mathbf{N}=(N_1,N_2,\dots)\in\N^{\N}$ be a strictly increasing sequence of integers such that \[
    \sigma^v \tau = \sigma^w \tau\quad\text{ for all $v,w\in \Sigma_{N_k}$ that satisfy
    $\mathbf{T}(v),\mathbf{T}(w)\neq\emptyset$.} \] We say that $N_k$ is a \emph{neck level} and
    that $\mathbf{N}$ is a \emph{neck list}.  \end{defn}

All that is left to describe is a measure of how individual relations are to be picked. Here we
consider a very general approach and all that we require are some properties of the measure with
respect to a dynamical system on $(\mathcal{T},\N^{\N})$ we call the \emph{neck shift}.
\begin{defn}
  Let $(\tau,\mathbf{N})\in(\mathcal{T}\times\N^{\N})$, where $\mathbf{N}$ is a strictly
  increasing sequence of natural numbers. We define the \emph{neck shift}
  $\Pi:(\mathcal{T}\times\N^{\N})\to(\mathcal{T}\times\N^{\N})$ by \[
  \Pi(\tau,\mathbf{N})=(\sigma^{1_{N_1}(\tau)}\tau,(N_2(\tau)-N_1(\tau),N_3(\tau)-N_1(\tau),\dots)),  \]
  where $1_{N_1}$ is the node at level $N_1$ consisting solely of $1$s.
\end{defn}

Whereas $\Prob$ was the product measure for random homogeneous systems above, we now only require
$\Prob$ to be an ergodic $\Pi$-invariant Borel probability measure such that the first neck has
finite expectation and necks are independent. 
\begin{defn}\label{defn:codetreemeasure} Let
  $(\tau,\mathbf{N})\in(\mathcal{T}\times\N^{\N})$, where $\mathbf{N}$ is a strictly increasing
  sequence of natural numbers, and let $\Pi$ be the neck shift. A \emph{code-tree measure} is any
  $\Pi$-invariant Borel measure on $(\mathcal{T}\times\N^{\N})$ such that \[ \E(N_1) =
  \int_{(\mathcal{T}\times\N^{\N})}N_1(\tau) \,\,\, d\Prob(\tau)<\infty,\]
  and for all open\footnote{The topology here is the product topology of the previously stated
    topology of $\mathcal{T}$ with the discrete topology on $\N^{\N}$.} subsets $A\subseteq ( \mathcal{T}\times \N^{\N} )$,  
  \begin{multline}\label{eq:independence}
    \Prob\{(\tau,\mathbf{N})\in A \}=\Prob\{\Pi(A)\}\cdot\Prob\Big\{ (\tau,\mathbf{N})\in(\mathcal{T}\times
    \N^{\N})\;\;\Big|\;\; \exists (\tau',\mathbf{N})\in A\text{ such that }\\\tau(v)=\tau'(v)\text{ for all
    }v\in\bigcup_{j=1}^{N_1(\tau)}\Sigma_j \Big\},
  \end{multline}
\end{defn}
We note that condition (\ref{eq:independence}) guarantees independence between neck levels, which
further implies strong mixing and ergodicity of the neck shift. Without loss of generality we assume
that all $(\tau,\mathbf{N})\in(\mathcal{T},\N^{\N})$ have infinitely many necks, since this set has
full measure with respect to the neck measure $\Prob$.

Clearly, the shift map $\sigma$ on $\Omega$ for random homogeneous attractors satisfies the
conditions of a neck measure and so
does the natural measure for $V$-variable sets.  In Section \ref{sect:generalcase} we will prove
that this model does not admit any gauge function but we will first consider simple reductions of
this model.

We end by referring the reader to \cite{Roy11} and \cite{Troscheit17} for other approaches using
random graphs which overlap with this model to some extend.

\section{Hausdorff and packing measure of random attractors}\label{sect:survey}
We often have to assume some conditions on the
overlaps of images in the iterated function systems to state meaningful dimension theoretic results.
In this article we will make use of the uniform
open set condition, but remark that some of the quoted results require slightly different overlap
conditions.

\begin{defn}[uniform open set condition (UOSC)]
Let $\mathbb{L}=\{\mathbb{I}_{\lambda}\}_{\lambda\in\Lambda}$ be a collection of IFSs. 
We say that $\mathbb{L}$ satisfies the \emph{uniform open set condition (UOSC)} if there exists an open set $\mathcal{O}$ such that 
\[f_\lambda^i(\mathcal{O})\subseteq\mathcal{O}\text{ and }f_\lambda^i(\mathcal{O})\cap f_\lambda^j(\mathcal{O})=\varnothing\text{ for all }\lambda\in\Lambda\text{ and }1\leq i,j \leq \cN_i\text{ where }i\neq j.
\]
\end{defn}\label{defn:UOSC}

One can easily determine the almost sure
Hausdorff dimension of these random attractors if one assumes the uniform open set condition and
similarity maps.
Recall that $c_\lambda^i$ is the contraction rate of $f_\lambda^i\in\mathbb{I}_\lambda$. The
Hausdorff (and packing) dimension of random homogeneous attractors is given, almost surely, by the unique $s$ satisfying
\[
\exp\E \left(\log\sum_{j=1}^{\cN_{\omega_{1}}}(c_{\omega_{1}}^{j})^{s}\right)=1,
\]
see \eg \cite{Hambly92, Roy11, Troscheit17}.
For random recursive sets the almost sure Hausdorff dimension is the unique $s$ satisfying
\[
\E\left( \sum_{j=1}^{\cN_{\omega_{1}}}(c_{\omega_{1}}^{j})^{s}\right)=1,
\]
see \eg \cite{Falconer86,Graf87}.
To ease notation we write $\fS^{s}_{\lambda}=\sum_{j=1}^{\cN_{\lambda}}(c_{\lambda}^{j})^{s}$ for $\lambda\in\Lambda$ and note that we assume

\begin{cond}\label{cond:boundedBelow}
Let $(\mathbb{L},\mu)$ be a random iterated function system. We assume that there exists $\cN$ such that $\cN_\lambda \leq \cN$ for all $\lambda\in\Lambda$ and there exist $0<c_{\min}\leq c_{\max}<1$ such that $c_{\min}\leq c_{i}^{j}\leq c_{\max}$ for all $i\in\Lambda$ and $j\in\{1,\dots,\mathbb{I}_{i}\}$. For the random recursive model we further assume $\E(\fS_{\lambda}^{0})>1$ and for the random homogeneous model we assume $\E(\log\fS_{\lambda}^{0})>0$.
\end{cond}
We note that Condition~\ref{cond:boundedBelow} implies that $c_{\min}^s\leq\fS_{\lambda}^{s}<\cN$
and $s\log c_{\min} \leq \log\fS_{\lambda}^{s}<\log \cN$ for all $s\geq 0$. We immediately obtain
that $\E(\fS_{\tau_{1}}^{s})<\infty$ (random recursive) and $\E(\log\fS_{\omega_{1}}^{s})<\infty$
(random homogeneous) for all $s\geq0$. 
Note that under these conditions we also have $\Var(\log\fS_{\omega_{1}}^{s})<\infty$ for all
$s\geq0$ for the random homogeneous model.

\begin{defn}
A random iterated function system $(\mathbb{L},\mu)$ is called \emph{almost deterministic} if there exists $s$ such that $\fS_{\lambda}^{s}=1$ for $\mu$-almost every $\lambda\in\Lambda$.
\end{defn}
If such $s$ exists it must necessarily be the almost sure Hausdorff dimension, \ie \[
s=\ess_{\tau\in\mathcal{T}}\dim_{H}(F_{\tau}).\] 

\begin{prop}[Graf \cite{Graf87}]
Let $(\mathbb{L},\mu)$ be a random iterated function system satisfying the UOSC and Condition~\ref{cond:boundedBelow} with associated random recursive set $F_{\tau}$ and write $s_0=\ess\dim_{H}F_{\tau}$. If $(\mathbb{L},\mu)$ is almost deterministic then
\[
0<\Haus^{s_0}(F_{\tau})<\infty \;\;\;\text{(a.s.)}
\]
and $\Haus^{s_0}(F_{\tau})=0$ (a.s.) otherwise.
\end{prop}

For random homogeneous attractors an analogous result holds. This is a special case of the one considered in \cite{Roy11}.
\begin{prop}[Roy and Urbanski~\cite{Roy11}]\label{thm:almostdeter}\index{uniform open set condition (UOSC)}
Let $(\mathbb{L},\mu)$ be a random iterated function system satisfying the UOSC and Condition~\ref{cond:boundedBelow} with associated random homogeneous set $F_{\omega}$ with almost sure Hausdorff dimension $s_0=\ess\dim_{H}F_{\omega}$.\index{Hausdorff dimension}
If $(\mathbb{L},\mu)$ is almost deterministic then
\[
0< \Haus^{s_0}(F_{\omega}) <\infty \;\;\;\text{(a.s.)}
\]
and $\Haus^{s_0}(F_{\omega})=0$ (a.s.) otherwise.
\end{prop}

In fact, more is known. If an attractor is not almost deterministic the packing measure of the random homogeneous attractor $F_\omega$ and the random recursive attractor $F_{\tau}$ is infinite almost surely, see Roy and Urbanski~\cite{Roy11} and Berlinkov and Urbanski~\cite{Berlinkov02}, respectively.

\subsection{Exact Hausdorff and packing measure for random recursive constructions}
Recall that a \emph{gauge function} $h:\R_0^+\to \R_0^+$ is a left-continuous, non-decreasing function such that $h(r)\to0$ as $r\to 0$. If there exists a constant $\lambda>1$ such that for all $x>0$ we have $h(2x)\leq\lambda h(x)$ we say that $h$ is \emph{doubling}. Recall the definition of the $h$-Hausdorff measure.
\begin{defn}
Let $F\subseteq\R^d$ and let $h$ be a gauge functions. The \emph{$h$-Hausdorff $\delta$-premeasure of $F$} is
\[
\Haus^h_\delta(F)=\inf\left\{\sum_{k=1}^\infty h(\lvert U_k\rvert)\;\;\Big|\;\; \{U_i\}\text{ is a countable $\delta$-cover of $F$}\right\},
\]
where the infimum is taken over all countable $\delta$-covers. The \emph{$h$-Hausdorff measure of $F$} is then
\[
\Haus^h(F)=\lim_{\delta\to 0}\Haus_\delta ^h(F).
\]
\end{defn}

Note that the gauge function need only be defined and non-decreasing on $[0,r_0]$ for some $r_0 >0$ since we are only concerned in its limit as $r\to 0$. Without loss of generality we shall assume that $h(t)\leq \overline{h}:=\min\{1,\sup h(s)\}$ where the supremum is taken over the largest interval where $h$ is defined and non-decreasing. Further we set $h(t)=\overline{h}$ for all $t>r_0$. For example, by writing $h(t)=t \log\log(1/t)$ we mean
\[
h(t)=\begin{cases}
\log t\log\log(1/t) &\text{for $t\leq r_0$}\\
r_0 \log\log(1/r_0)&\text{for $t>r_0$}
\end{cases},
\] 
where $r_0$ is the unique solution to $\log\log(1/r_0)=(\log(1/r_0))^{-1}$, its unique stationary point.

For the random recursive case, Graf, Mauldin, and Williams determined the natural gauge that gives positive and finite Hausdorff measure.
\begin{theo}[Graf, Mauldin, and Williams~\cite{Graf88, Mauldin87}]
  Let $(\mathbb{L},\mu)$ be a random iterated function system that is not almost deterministic. Let $F_\tau$ be the associated random recursive attractor.
Assume that 
\[
\E\left(\sum_{j}(c_{\omega_{1}}^{j})^{0}\right)>1.
\]
Let 
\begin{equation}
h_{\beta}^{s}(t)=t^{s}(\log\log(1/t))^{1/\beta} \text{ and }\beta_{0}=\sup\left\{\beta \;\Big|\; \sum_{j}(c_{\omega_{1}}^{j})^{s/(1-1/\beta)}\leq1 \text{(a.s.)}\right\}.
\end{equation}
Then, $\Haus^{h_{\beta}^{s_0}}(F_{\tau})<\infty$ for all $\beta>\beta_{0}$, where $s_0=\ess\dim_{H}F_{\tau}$. 
\end{theo}

The authors then proceed to give technical conditions under which $\beta_{0}=1-s/d$, where $d$ is the dimension of the ambient space. Under these conditions the $h_{\beta_{0}}^{s}$-Hausdorff measure of $F_{\tau}$ is positive and finite almost surely. 
Checking the conditions one obtains that Mandelbrot percolation\index{Mandelbrot percolation}\index{percolation} of $[0,1]^{d}$ has positive and finite measure at this critical value $\beta_{0}$.

Liu~\cite{Liu00} investigated the Gromov boundary of Galton-Watson trees with i.i.d.\ randomised descendants.
Let $m=\E(N)$, where $N$ is the number of descendants, $\alpha=\log m$, and assume that $\E(N\log N)<\infty$. 
If $\overline m=\ess\sup N<\infty$, then the appropriate gauge function for which one obtains positive and finite measure of the boundary (with respect to a natural metric) is 
\[
h(t)=t^{\alpha} (\log\log (1/t))^{\beta}, \text{ where }
\beta=1-\frac{\log m}{\log \overline m}.
\]
For the packing measure to be positive and finite the appropriate gauge function is
\[
h^*(t)=t^{\alpha} (\log\log (1/t))^{\beta^*}, \text{ where }
\beta^*=1-\frac{\log m}{\log \underline m},
\]
with $\underline m=\ess\inf N>1$

Berlinkov and Mauldin~\cite{Berlinkov02} provide the following, more general result for the packing
measure of random recursive sets.\index{packing measure}
Under the same \emph{almost deterministic} condition they show that the $s$-dimensional packing
measure\index{packing measure} is positive and finite almost surely. When this fails, the packing
measure is $\infty$ almost surely, assuming the UOSC in both cases.\index{uniform open set condition
(UOSC)}
Let $s$ denote the almost sure packing dimension.\index{packing dimension} The authors prove that for the gauge function
\[
h_{\beta}^{s}(t)=t^{s}(\log\log(1/t))^{\beta}, \text{ where $\beta$ satisfies
}0<\liminf_{a\to0}-a^{-1/\beta}\log\Prob_T(\fS_\lambda^s<a)<\infty,
\]
the packing measure is almost surely finite. 
We remark that the constant $\beta$ may not exist and only coincides with the $\beta_0$ in the Hausdorff measure statement in trivial cases.

Additionally, Berlinkov and Mauldin give an integral test \cite[Theorem 6]{Berlinkov02} to determine whether the packing measure is $0$ almost surely.
They further conjecture a lower bound that Berlinkov proved in \cite{Berlinkov03}:
If the random variable\index{random variable (r.v.)} $\fS_\lambda^s$ is of exponential type, \ie if
\[
C^{-1}a^{1/\beta}\leq-\log\Prob_T(0<\fS_\lambda^s\leq a)\leq Ca^{1/\beta}
\]
for some $C,\beta>0$ and all $a\in(0,1)$, then the packing measure is positive and finite almost
surely with gauge function $h_{\beta}^{s}(t)$.

\section{Bounds on the gauge function for random homogeneous
constructions}\label{sect:boundshomogeneous}
It is of course of interest to determine the gauge functions for which one obtains positive and
finite measure for random homogeneous systems. In particular, self-similar and self-conformal sets
that satisfy the open set condition have positive and finite Hausdorff measure.
One might expect that random homogeneous are of a
similar nature and that a gauge function of the form $t^s (\log\log(1/t))^\beta$ for some
exponent $\beta$ should work for all natural random code-tree constructions.

However, we will show that this turns out not to be the case. Indeed, for $s=F_\omega$, we first
prove better bounds on the fine dimension, \ie bounds on $h$ that give positive or finite measure.
In the next Section we show that there is no gauge function that gives positive and finite measure,
but the bounds established here are still of separate interest. We will argue that
\[
h_{1}(t)=t^{s}\exp\left(\sqrt{(\log (1/t))(\log \log \log (1/t))}\right),
\]
gives good bounds on the positivity and finiteness of random homogeneous constructions.
Let $\beta,\gamma\in\R$, we similarly define
\[
h_{1}(t,\beta,\gamma)=t^{s}\exp\left(\sqrt{2\beta(\log (1/t))(\log \log \log (1/t^{\beta}))}\right)^{1-\gamma}.
\]

\subsection{Equicontractive homogeneous random attractors}

The first thing to note is that $h_1(t,\beta,\gamma)$ is doubling in $t$.\index{doubling}
\begin{lma}\label{lma:doubling}
Fix $\beta,\gamma>0$. There exists $t_0,\rho>0$ such that 
\[
h_1(t,\beta,\gamma)\leq \rho h_1(2t,\beta,\gamma)\leq \rho^2 h_1(t,\beta,\gamma)
\]
for all $0<t<t_0$.
\end{lma}
\begin{proof}
Let $\kappa\in\R$ and 
\[
h_*(x+\kappa)=\sqrt{\beta (x+\kappa) \log\log( \beta(x+\kappa))}.
\]
This is well defined for $\log\log\beta (x+\kappa)>1 \implies x>e^e /\beta-\kappa$.
It can easily be seen that this function is strictly increasing in $x$, and differentiating we obtain,
\[
h'_*(x+\kappa)=\sqrt{\beta}\cdot\frac{1/(\log(\beta(x+\kappa))+\log\log(\beta(x+\kappa))}{2\sqrt{(x+\kappa)\log\log(\beta(x+\kappa))}}.
\]
Then, for $\kappa>0$,
\begin{align*}
\frac{h'_*(x)}{\sqrt{\beta}}&=\frac{1/(\log(\beta x)+\log\log(\beta x)}{2\sqrt{ x\log\log(\beta x)}}\\
&> \frac{1/(\log(\beta(x+\kappa))+\log\log(\beta(x+\kappa))}{2\sqrt{x\log\log(\beta(x+\kappa))}}= \frac{h'_*(x+\kappa)}{\sqrt{\beta}}
\end{align*}
and so $h'_*(x+\kappa)-h'_*(x)<0$ and $h_*(x+\kappa)-h_*(x)$ is decreasing, \ie there exists some $\rho_0$ such that 
\[
0\leq h_*(x+\kappa)-h_*(x) \leq \rho_0.
\]
Now substituting $\kappa=-\log 2$ and $x=-\log t$, \ie $x+\kappa=\log(1/2t)$, we obtain,  for $0<t<t_0$ and $t_0>0$ small enough,
\[
0\leq \sqrt{\beta\log(1/(2t)\log\log(\beta\log1/(2t))}- \sqrt{\beta\log(1/t)\log\log(\beta\log1/t)}\leq\rho_0,
\]
and
\begin{align*}
2^s &\leq \frac{(2t)^s}{t^s}\cdot e^{(1-\gamma)\sqrt{2}\cdot\left(
\sqrt{\beta\log(1/(2t)\log\log(\beta\log1/(2t))}- \sqrt{\beta\log(1/t)\log\log(\beta\log1/t)}
\right) }\\
&\leq 2^s e^{(1-\gamma)\sqrt{2}\rho_0}.
\end{align*}
But then 
\[
2^s \leq \frac{h_1(2t,\beta,\gamma)}{h_1(t,\beta,\gamma)}\leq 2^s e^{(1-\gamma)\sqrt{2}\rho_0},
\]
as required.
\end{proof}

We require some further results on homogeneous systems satisfying the UOSC. Let $\varepsilon>0$ and
$\mathcal{O}$ be the open set guaranteed by the UOSC. We define $\Xi_\varepsilon(\omega)$ be the words in
$\mathbf{C}_\omega^*$ such that $\lvert f_e(\mathcal{O})\rvert\leq\epsilon$, but $\lvert
f_{e^{-}}(\mathcal{O})\rvert>\varepsilon$ for all $e=e_1 e_2 \dots e_k
\in\Xi_{\varepsilon}(\omega)$, where $e^{-}=e_1 e_2 \dots e_{k-1}$.
\begin{lma} \label{upperasslem1}
Assume that $(\mathbb{L},\mu)$ satisfies the UOSC. 
Then
\[
\#\{e\in\Xi_{\epsilon}(\tau)\mid \overline{f_e(\mathcal{O})}\cap B(z,\varepsilon)\neq\varnothing\} \, \leq \,  (4/c_{\min})^{d}
\]
for all $z\in F_{\tau}$ and  $\epsilon \in (0,1]$, where $\mathcal{O}$ is the open set guaranteed by the UOSC.
\end{lma}
A proof for this lemma can be found in~\cite[Lemma 5.1.5]{TroscheitPhDThesis} and \cite{Troscheit17a}.

\vskip1em

For ease of exposition we deal with the basic case where all maps in a fixed IFS contract equally. 
Recall that we assume $\E(\fS_{\omega_1}^0)>1$ throughout.
\begin{theo}\label{thm:uppergauge}\index{Hausdorff measure}\index{uniform open set condition (UOSC)}
Let $F_{\omega}$ be the random homogeneous attractor associated to the RIFS $(\mathbb{L},\mu)$ satisfying the UOSC and suppose that $c_{\lambda}^{i}=c_{\lambda}\in[c_{\min},c_{\max}]$ for every $i\in\{1,\dots,\#\mathbb{I}_\lambda\}$ and $\lambda\in\Lambda$, where $0<c_{\min} \leq c_{\max}<1$.
Let $\epsilon>0$, $s=\ess\dim_H F_\omega$ and $\beta=\Var(\log\fS_{\omega_{1}}^{s})/\eta$ for some $\eta\in\R$ (arising in the proof), then 
\[
\Haus^{h_{1}(t,\beta,\epsilon)}(F_{\omega})=0,
\]
almost surely.
\end{theo}

To prove this we need the \emph{Law of the iterated logarithm (LIL)}, see \eg Athreya and
Lahiri~\cite{AthreyaMeasure}. 

\begin{prop}[Law of the iterated logarithm (LIL)]
 \label{thm:LIL} Let $\{X_i\}_{i\in\N}$ be a sequence of
i.i.d.\ random variables on a probability space $(\Omega,P)$ with mean $m_0=\int_{\Omega}X_0
dP(\omega)$ and variance $V_0=\Var(X_0)$. Then, almost surely, \[ \limsup_{k\to\infty}\frac{\sum_{i=1}^k (X_i -
m_0)}{\sqrt{(2V_0 k \log\log V_0 k)}}=1 \] and similarly \[ \limsup_{k\to\infty}\frac{\sum_{i=1}^k
(X_i - m_0)}{\sqrt{(2V_0 k \log\log V_0 k)}}=1.  \] \end{prop}

\begin{proof}[Proof of Theorem~\ref{thm:uppergauge}]
Let $\mathcal{O}$ be the open set guaranteed by the UOSC, we assume without loss of generality that $\lvert\mathcal{O}\rvert=1$.  From the definition of Hausdorff measure
\[
\Haus^{h_{1}(t,\beta,\epsilon)}(F_{\omega})\leq\sum_{e\in\mathbf{C}_{\omega}^k}
h_{1}\left(\lvert\overline{ f_e(\mathcal{O})}\rvert,\beta,\epsilon\right)
\]
for all $k\in\N$.
So, writing $v=\Var(\log\fS_{\omega_{1}}^s)$,
\begin{align*}
\Haus^{h_{1}(t,\beta,\epsilon)}(F_\omega)&\leq \liminf_{k\to\infty}\sum_{e\in\mathbf{C}_{\omega}^k}
h_{1}\left(\lvert \overline{f_e(\mathcal{O})}\rvert,\beta,\epsilon\right)\\
&=\liminf_{k\to\infty}\left(\prod_{i=1}^{k}\#\mathbb{I}_{\omega_i}\right)(c_{\omega_1}c_{\omega_2}\dots
c_{\omega_k})^{s}\\
&\hspace{0.8cm}\cdot\exp\left(\sqrt{2\beta\log(1/(c_{\omega_1} \dots c_{\omega_k}))\log\log\beta\log(1/(c_{\omega_1} \dots c_{\omega_k}))}\right)^{1-\epsilon}\\&\\
&=\liminf_{k\to\infty}\exp\left[\left(\sum_{i=1}^{k}\log\fS_{\omega_i}^{s}\right)\right.\\
&\hspace{3cm}+(1-\epsilon)\sqrt{2k\beta\log(C_{\omega}^k)\log\log(\beta k\log(C_{\omega}^k))}\Bigg]\\
\intertext{for $C_{\omega}^k=(c_{\omega_1}c_{\omega_2}\dots c_{\omega_k})^{-1/k}$, and so}
&=\liminf_{k\to\infty}\exp\Bigg[\left(\sum_{i=1}^{k}\log\fS_{\omega_i}^{s}\right)\\&\\
&\hspace{2.8cm}+(1-\epsilon)\sqrt{2\frac{\log C_{\omega}^k}{\eta}kv\log\log\left(\frac{\log C_{\omega}^k}{\eta}kv\right)}\Bigg].
\end{align*}
Note that we can apply the law of the iterated logarithm, Proposition~\ref{thm:LIL}, to sums over
the random variables $Y_i=\log\fS_{\omega_i}^{s}$ where $Y_i$ are i.i.d.\ with $\E(Y_1)=0$ and $0<\Var(Y_1)<\infty$. Thus
\[
\Prob_H\left\{ \sum_{i=1}^k Y_i \leq -(1-\epsilon/2)\sqrt{2vk\log\log(vk)}\text{ for infinitely many }k\in\N\right\}=1
\]
Let $(i_1,i_2,\dots)$ be a sequence of indices where the above inequality holds. Note that $c_{\min}\leq C_{\omega}^k<c_{\max}$ for all $\omega$ and $k$, and so $\log c_{\min}\leq \log C_{\omega}^k<\log c_{\max}$. Therefore, for some uniform $\widetilde{\eta}\in [\log c_{\min},\log c_{\max}]$, we have $\log  C_{\omega}^{i_k}/\widetilde{\eta}\geq 1$ for infinitely many $k$, for almost all $\omega$. We can thus choose $\eta$ the greatest value for which this is satisfied.

We get, almost surely,
\[
\Haus^{h_{1}(t,\beta,\epsilon)}(F_\omega)\leq\lim_{k\to\infty}\exp\left(-\frac{\epsilon}{3}\sqrt{2\Var(\log\fS_{\omega_{1}}^{s})k\log\log(\Var(\log\fS_{\omega_{1}}^{s})k)}\right)=0,
\]
completing the proof.
\end{proof}
We note that if $c_\lambda=\widetilde c$ for every $\lambda$, then $\eta=\log \widetilde c$. Note
also the following corollary which implies that the `fine dimension', \ie the dimension according to
the gauge function, is distinct from the random recursive case.
\begin{cor}
Let $F_{\omega}$ be the random homogeneous attractor associated to the RIFS $(\mathbb{L},\mu)$
satisfying the UOSC  and suppose that $c_{\lambda}^{i}=c_{\lambda}\in[c_{\min},c_{\max}]$ for every
$i\in\{1,\dots,\#\mathbb{I}_\lambda\}$ and $\lambda\in\Lambda$, where $0<c_{\min} \leq c_{\max}<1$.
Let $h_{\beta}^{s}(t)=t^{s}(\log\log(1/t))^{\beta}$, where $s=\ess\dim_H F_{\omega}$ and $\beta>0$,
then \[\Haus^{h_{\beta}^{s}(t)}(F_\omega)=0.\quad\text{(a.s.)}\]
\end{cor}
\begin{proof}
We check
\begin{align*}
\lim_{t\to0}\frac{h_{\beta}^{s}(t)}{h_1(t,\beta',\epsilon)}&=\lim_{t\to0}\frac{t^{s}(\log\log(1/t))^{\beta}}{t^{s}\exp\left(\sqrt{2\beta'(\log
  (1/t))(\log \log \log (1/t^{\beta'}))}\right)^{1-\epsilon}}\\
&=\lim_{t\to0}\frac{(\log\log(1/t))^{\beta}}{\exp\left((1-\epsilon)\sqrt{2\beta'(\log (1/t))(\log
  \log \log (1/t^{\beta'}))}\right)}\\
&\leq\lim_{t\to0}\frac{(\log(1/t))^{\beta}}{\exp\left((1-\epsilon)\sqrt{2\beta'(1/t)}\right)}=0.
\end{align*}
This holds for all $\beta,\beta',\epsilon>0$ and the behaviour of the limits is sufficient for the
desired result.
\end{proof}

Considering $h_{1}(t,\beta,-\epsilon)$ the law of the iterated logarithm guarantees a similar lower bound where the sum diverges. We will then use the mass distribution principle for the $h$-Hausdorff measure to establish infinite Hausdorff measure for $\epsilon>0$.

\begin{theo}[Mass Distribution Principle]\label{thm:massDistribution}
Let $\mu$ be a finite measure supported on $F$ and suppose that for some gauge function $h$ there are constants $c>0$ and $r_0$ such that $\mu(U)\leq c h(\lvert U\rvert)$ for all sets $U$ with $\lvert U\rvert<r_0$. Then $\Haus^h(F)\geq \mu(F)/c$.
\end{theo}
While the proof can be found in a number of places (\eg \cite{Olsen03a}), we recall it for completeness.
\begin{proof}
Consider any countable open cover $\{O_i\}$ of $F$.
Then
\[
\mu(F)\leq\mu\left(\bigcup_i O_i\right)\leq\sum_i \mu(O_i)\leq c\sum_i h(\lvert O_i \rvert).
\]
But then, taking the infimum for each $\delta>0$, we have $\Haus^h (F)\geq\Haus^h_{\delta}(F)\geq \mu(F)/c$.
\end{proof}

\begin{theo}\label{thm:lowergauge}\index{Hausdorff measure}
Let $F_{\omega}$ be the random homogeneous attractor associated to the RIFS $(\mathbb{L},\mu)$ satisfying the UOSC  and suppose that $c_{\lambda}^{i}=c_{\lambda}\in[c_{\min},c_{\max}]$ for every $i\in\{1,\dots,\#\mathbb{I}_\lambda\}$ and $\lambda\in\Lambda$, where $0<c_{\min} \leq c_{\max}<1$. 
Let $\epsilon>0$, $s=\ess\dim_H F_\omega$ and $\beta_0=\eta_0\Var(\log\fS_{\omega_{1}}^{s})$ for some $\eta_0\in\R$ (arising in the proof), then 
\[
\Haus^{h_{1}(t,\beta_0,-\epsilon)}(F_{\omega})=\infty
\]
holds almost surely.
\end{theo}
\begin{proof}
We use the same notation of the proof of Theorem~\ref{thm:uppergauge}. Let $\epsilon>0$ be given and write $v=\Var(Y_1)$. Then the law of the iterated logarithm, Theorem~\ref{thm:LIL}, implies
\[
\Prob_H\left\{ \sum_{i=1}^k Y_i \leq -(1+\epsilon)\sqrt{2vk\log\log(vk)}\text{ for infinitely many }k\in\N\right\}=0
\]
and so, writing $D_k (\omega)=(c_{\omega_1} c_{\omega_2}\dots c_{\omega_k})^{-1}$, and $\eta_k(\omega)=k/\log D_k(\omega)$,
\begin{multline}\label{eq:bddproduct}
\Prob_H\bigg\{ \fS_{\omega_1}^s \fS_{\omega_2}^s \dots \fS_{\omega_k}^s \geq C\exp\Big(-(1+\epsilon)\sqrt{2v\eta_k(\omega)\log( D_k(\omega))}\\
\cdot\sqrt{\log\log(v\eta_k(\omega)\log(D_k(\omega))}\Big)\text{ for all }k\geq l_0(\omega)\text{ where }l_0(\omega)\in\N\bigg\}=1
\end{multline}
for any $C\in\R$.
Since $c_\lambda$ is bounded away from $0$ and $1$, the sequence $\eta_k(\omega)$ is uniformly bounded in $k$ and $\omega$. Therefore there exists uniform $\eta_0$ such that~(\ref{eq:bddproduct}) holds with $\eta_k(\omega)$ replaced by $\eta_0$.
Then, on a full measure set,
\begin{align}
  \left( \prod_{i=1}^k \cN_{\omega_i}\right)&\geq C{D_k(\omega)^s}\exp\Big(-(1+\epsilon)\sqrt{2v\eta_0\log( D_k(\omega))\log\log(v\eta_0\log(D_k(\omega))}\Big)\nonumber\\
&=\frac{C}{h_1(D_k(\omega)^{-1},\beta_0,-\epsilon)}\label{eq:usfulbound1}
\end{align}
holds for all $k\geq l_0(\omega)$.

We define a random measure $\nu_{\omega}$ on $F_\omega$. Assume $e\in\mathbf{C}_\omega^k$ for some $k\in\N$, for every basic cylinder\index{cylinder} we set 
\[
  \widetilde{\nu}_\omega(f_e(\mathcal{O}))=\left( \prod_{i=1}^k \cN_{\omega_i}\right)^{-1}.
\]
This extends to a unique random measure $\nu_\omega$ on $F_\omega$ for every $\omega\in\Omega$ by Carath\'eodory's extension theorem.\index{Carath\'eodory's extension theorem}
We now show that, almost surely, there exists $C_\omega>0$ such that $\nu_\omega(U)\leq (C_\omega/C) h_1(\lvert U\rvert ,\beta_0,-\epsilon)$ for all small enough open $U$ that intersect $F_\omega$. Let $U$ be such that $u=2\lvert U\rvert<(c_{\min})^{l_0(\omega)}$ and choose $z\in (U\cap F_\omega)$, then
\begin{align*}
\nu_\omega(U)&\leq \nu_\omega(B(z,u))\leq \nu_{\omega}\left(\bigcup_{\substack{e\in\Xi_{u}(\omega)
\\ \overline{f_e(\mathcal{O})}\cap B(z,u)\neq\varnothing}}f_e(\mathcal{O})\right)\\ &\\
&=\sum_{\substack{e\in\Xi_{u}(\omega) \\ \overline{f_e(\mathcal{O})}\cap
B(z,u)\neq\varnothing}}\hspace{-.5cm}\left( \prod_{i=1}^{k(u)} \cN_{e_i}\right)^{-1}
\leq \left(\frac{4}{c_{\min}}\right)^d\left( \prod_{i=1}^{k(u)} \cN_{e_i}\right)^{-1},
\end{align*}
by Lemma~\ref{upperasslem1}, where $k(u)$ is the common length of all $e\in \Xi_u(\omega)$. Note that by assumption $k(u)\geq l_0(\omega)$. Therefore, using (\ref{eq:usfulbound1}),
\[
\nu_\omega(U)\leq  \left(\frac{4}{c_{\min}}\right)^d {C}^{-1}  h_1 (D_{k(u)}(\omega)^{-1},\beta_0,-\epsilon).
\]
Recall that $h_1$ is doubling, $c_{\omega_1}c_{\omega_2}\dots c_{\omega_k}=D_k(\omega)<u$, and so there exists $\kappa>0$ such that 
\[
\nu_\omega(U)\leq  (\kappa/C)  h_1 (\lvert U\rvert,\beta_0,-\epsilon).
\]
Now, using the mass distribution principle, Theorem~\ref{thm:massDistribution}, we conclude 
\[
\Haus^{h_1 (t,\beta_0, -\epsilon)}(F_\omega)\geq \frac{C}{\kappa}.
\]
The desired conclusion follows from the fact that $C$ was arbitrary.
\end{proof}
Note that the constants $\eta$ and $\eta_0$ might not coincide, and thus $\beta=\beta_0$ might not hold.

Both Theorems~\ref{thm:uppergauge} and \ref{thm:lowergauge} seems to suggest that $h_1(t,\beta,0)$ with $\beta=\beta_0$ is the correct function that gives positive and finite Hausdorff measure. However, we will now show that not only is there no $\beta$ such that 
\[
0<\Haus^{h_1(t,\beta,0)}(F_\omega)<\infty,
\]
but there does not exist any gauge function $h$ that gives positive and finite Hausdorff measure.

\section{The non-existence of gauge functions for self-similar code-trees with
necks}\label{sect:nonExistenceMain}
In this section we will prove our main result, namely that random code-trees with necks and 
random homogeneous attractors do not admit a gauge function, provided they are not almost deterministic.
First, we will prove the equicontractive random homogeneous case as its proof is considerably
simpler than the full case. Further, the simplistic equicontractive case gives a better intuition as
to why the gauge function cannot exist, something that can more easily get lost in full generality.

Before we deal with these we first show that we can without loss of generalisation assume $h(t)$ is
of the form $h(t)=t^{s+g(t)}$ where $g(t)\nearrow 0$ from below as $t\to 0$.
Consider a self-similar code-tree with necks. By Condition~\ref{cond:boundedBelow} the attractor has
positive Hausdorff dimension. Since having Hausdorff dimension greater than $\alpha>0$ is a tail
event, ergodicity implies that there exists an almost sure Hausdorff dimension $s$. Thus
$\log(h(t))/\log(t)\to s$ and $g(t)\to 0$. Further, $\Haus^s (F_\tau)=0$ almost surely and we must
have $g(t)\leq 0$ for positive $h$-Hausdorff measure. Lastly, we can assume $g(t)$ to be increasing
as there is no `preferred' scale, \ie the contractions that are randomly chosen are bounded away
from $0$ and $1$ and the expected block size is finite, and having positive and finite measure is a
tail event with probability $0$ or $1$.
\subsection{Equicontractive homogeneous random attractors}\label{sect:equicontractive}
Let us first consider the case when $(\mathbb{L},\mu)$ satisfies the UOSC and all maps in all IFSs have the same contraction rate $c\in(0,1)$.
Then, 
\[
\sum_{e\in \mathbf{C}_{\omega}^k}h(c)=
\left(\prod_{i=1}^k \cN_{\omega_{i}}\right)h(c^k).
\]
As we will show below, finding a gauge function for positive and finite Hausdorff measure reduces to the problem of finding a gauge function such that 
\[
\liminf_{k\to\infty}\left(\prod_{i=1}^k \cN_{\omega_{i}}\right)h(c^k)\in(0,\infty)
\]
almost surely.
But, for $h_k=\log h(c^k)$,
\[
\liminf_{k\to\infty}\left(\prod_{i=1}^k \cN_{\omega_{k}}\right)h(c^k)=
\liminf_{k\to\infty}\exp\left(h_k+\sum_{i=1}^k \log \cN_{\omega_{i}}\right).
\]
Therefore, the problem of finding a  suitable gauge function becomes equivalent to finding a (fixed)
sequence $(h_k)$
such that $\liminf_{k\to\infty} h_k+\sum_{i=1}^k X_i$ is positive and finite for some sequence of
i.i.d.\ random variables $X_i$ with positive variance. Something that clearly does not exist by the
Central Limit Theorem.

Since the general case is somewhat more complex to deal with, we will first prove the same result
for equicontractive sets using methods that will generalise more readily.
Throughout the remainder of this section we assume that $(\mathbb{L},\mu)$ satisfies the UOSC and is equicontractive, \ie for every $\lambda\in\Lambda$ all maps $f_\lambda^i\in\mathbb{I}_\lambda$ have contraction ratio $c_\lambda^i=c_\lambda$.
This is slightly more general than the case considered just above.
We will continue to assume Condition~\ref{cond:boundedBelow} and that $(\mathbb{L},\mu)$ is not almost deterministic. In particular, this means that there exists $0<\gamma<1$ and $\epsilon>0$ such that
\begin{equation}\label{eq:gap}
p_0:=\mu\Big\{ \sum_{f_\lambda^i\in\mathbb{I}_\lambda} (c_\lambda)^{s-\epsilon}\leq \gamma\Big\}>0
\end{equation}
where $s$ is the almost sure Hausdorff dimension of the homogeneous random attractor $F_\omega$ with
$\omega\in\Omega=\Lambda^{\N}$ and $\Prob_H=\mu^{\N}$.

Let $h$ be an arbitrary gauge function. Given $\omega\in\Omega$ we obtain
\[
\sum_{e\in \mathbf{C}_{\omega}^k}h(c_e)=
\sum_{e\in \mathbf{C}_{\omega}^k}h(c_{\omega_1}c_{\omega_2}\dots c_{\omega_{k}})=\left(\prod_{i=1}^k \cN_{\omega_{i}}\right)h(c_{\omega_1}c_{\omega_2}\dots c_{\omega_{k}}).
\]

The equicontractive case is simpler for two reasons. First, it allows us to write the sum of gauge values as the product of the number of descendants with the gauge function at a single value. Further, the equicontractive property also implies that the Hausdorff measure is given by the lower limit of the sums of gauge functions values.

\begin{lma}\label{lma:levelsetsareequal}
Let $(\mathbb{L},\mu)$ be a RIFS that satisfies the UOSC, Condition~\ref{cond:boundedBelow}, is
equicontractive, and let $F_\omega$ be the associated random homogeneous attractor.
Then, for any gauge function $h$, and all $\omega\in\Omega$ there exists $\kappa_\omega>0$ such that
\begin{equation}\label{eq:equalLevels}
\Haus^h (F_\omega)=\kappa_\omega \liminf_{k\to\infty}\sum_{e\in \mathbf{C}_{\omega}^k}h(c_e).
\end{equation}
\end{lma}
\begin{proof}

Let $\mathcal{O}$ be the open set guaranteed by the UOSC. Clearly,
\[
F_\omega\subseteq \bigcup_{e\in\mathbf{C}_\omega^k} \overline{f_e(\mathcal{O})}
\]
and so 
\[
\Haus^h (F_\omega)\leq \kappa_\omega \liminf_{k\to\infty}\sum_{e\in \mathbf{C}_{\omega}^k}h(c_e)
\]
for some $\kappa_\omega$. For the lower bound, we see that
$\{\overline{f_e(\mathcal{O})}\}_{e\in\mathbf{C}_\omega^*}$ is a Vitali cover. A standard argument then gives that for every $\epsilon>0$ there exist subsets of $C_k\subset\mathbf{C}_\omega^*$ such that for every $v,w\in C_k$, neither $v$ is a parent of $w$ or vice versa, and that every $v\in C_k$ has length at least $k$. Using Lemma~\ref{upperasslem1}, we can extend $C_k$ to a full tree by admitting at most $(4/c_{\min})^d$ many further cylinders. Thus, for some set $C'_k$,
\[
\frac{1}{(4/c_{\min})^d}\sum_{e\in C'_{k}}h(c_e \cdot\lvert\mathcal{O}\rvert)\leq \Haus^h(F_\omega)+\epsilon.
\] 
Now note that by homogeneity, we can assume that all $e\in C'_{k}$ must be at the same tree level
$k$. This follows from the observation that, if $e$ and $f$ are two different nodes in levels
$k_e>k_f$ then there exists a parent of $e$, say $e'$, in level $k_f$ that does not have $f$ as a
descendant. Further, there exists a (non-trivial) collection of descendants $e_i$ such that $\sum
h(c_{e_i})\sim h(c_f)$. But since all the descendants of $f$ behave exactly as the descendants of
$e'$, we do not need to consider the children of $e'$ for an efficient cover. By induction this
extends to a comparable cover over a single tree level and 
\[
\kappa_\omega\liminf_{k\to\infty}\sum_{e\in \mathbf{C}_{\omega}^k}h(c_e)\leq \Haus^h(F_\omega)+\epsilon
\]
for some $\kappa_\omega>0$. Since $\epsilon$ was arbitrary, the desired conclusion follows.
\end{proof}

We are now ready to state and prove our main theorem of this section.

\begin{theo}\label{thm:mainequicontractive}
Let $(\mathbb{L},\mu)$ be a RIFS that satisfies the UOSC, Condition~\ref{cond:boundedBelow}, is
equicontractive, and let $F_\omega$ be the associated random homogeneous attractor.
Let $h$ be a gauge function, then for almost all $\omega\in\Omega$,
\[
\Haus^h(F_\omega)\in \{0,\infty\}.
\]
In particular, there exists no gauge function such that the $h$-Hausdorff measure is positive and finite almost surely.
\end{theo}
\begin{proof}
In light of Lemma~\ref{lma:levelsetsareequal} all that remains is to show that the right hand side of (\ref{eq:equalLevels}) is $0$ or $\infty$ almost surely, irrespective of $h$. Now $\{0<\Haus^h(F_\omega)<\infty\}$ is a tail-event and thus has probability $0$ or $1$ by Kolmogorov's Zero -- One Law. Therefore, we assume for a contradiction that 
\[
\Prob_H\Big\{ \omega\in\Omega \;\Big|\; \liminf_{k\to\infty}\sum_{e\in \mathbf{C}_{\omega}^k}h(c_e)\in(0,\infty)  \Big\}=1.
\]
Recall equation (\ref{eq:gap}), which implies that there exists positive probability $p_0$ such that
a letter $\lambda$ is picked that lowers the sum. Call the set of letters $B$. 
Using the assumption that $h(t)=t^{s+g(t)}$ for some decreasing $g(t)\leq0$, let $\epsilon>0$ be as given in (\ref{eq:gap}) and choose $t_0$ such that $g(t)\geq-\epsilon$ for all $0\leq t<t_0$. Then,
\begin{align}
\sup_{\lambda\in B}\sum_{f_\lambda^i\in\mathbb{I}_\lambda}h(c_\lambda\cdot t)&=
\sup_{\lambda\in B}\sum_{f_\lambda^i\in\mathbb{I}_\lambda}(c_\lambda\cdot t)^{s+g(c_\lambda\cdot t)}\nonumber\\
&\leq \sup_{\lambda\in B}\sum_{f_\lambda^i\in\mathbb{I}_\lambda}(c_\lambda\cdot t)^{s+g(t)}\nonumber\\
&\leq h(t)\;\sup_{\lambda\in B}\sum_{f_\lambda^i\in\mathbb{I}_\lambda}(c_\lambda)^{s-\epsilon}\leq \gamma \,h(t)\label{eq:gapBound}
\end{align}
for some $0<\gamma<1$ as given in (\ref{eq:gap}).

For $i\in\Z$, we define level sets $E_i$ by
\[
E_i := \Big\{ \omega\in\Omega \;\Big|\; \liminf_{k\to\infty}\sum_{e\in \mathbf{C}_{\omega}^k}h(c_e)\in(\gamma^{i+1},\gamma^i\,]  \Big\}
\]
Clearly, $\bigcup_{i\in\Z} E_i$ is a disjoint union with full measure and there exists $j_0\in\Z$
such that $\Prob_H(E_{j_0})>0$. Since $c_{\max}<1$, there exists $k_0$ such that $c_e<t_0$ for all
$e\in\mathbf{C}_\omega^k$, where $k>k_0$. For $\omega\in\bigcup_{i\in\Z} E_i$ we write
$\zeta_k(\omega)$ for the $k$-th time of being close to the lower limit, \ie
\[
\zeta_1(\omega)=\min\Big\{k>k_0 \;\Big|\; \sum_{e\in \mathbf{C}_{\omega}^k}h(c_e)\leq \gamma^i\text{ and }\sum_{e\in \mathbf{C}_{\omega}^l}h(c_e)>\gamma^{i+1}, (\forall l\geq k), \text{ where }\omega\in E_i \Big\}
\]
and
\begin{multline}
\zeta_j(\omega)=\min\Big\{k> \zeta_{j-1}(\omega) \;\Big|\; \sum_{e\in \mathbf{C}_{\omega}^k}h(c_e)\leq \gamma^i\text{ and }\\\sum_{e\in \mathbf{C}_{\omega}^l}h(c_e)>\gamma^{i+1}, (\forall l\geq k), \text{ where }\omega\in E_i \Big\}\nonumber
\end{multline}
Recall that positive and finite Hausdorff measure is a tail event. Thus, inserting a letter almost
never changes the Hausdorff measure from a positive and finite value to one in $\{0,\infty\}$. We
therefore assume, without loss of generality, that the insertion of a letter does not change the
word belonging to the tail event $\{0<\Haus^h(F_\omega)<\infty\}$.

Now consider $\widehat{E}_{j_0}(i)$, where
\begin{multline}
\widehat{E}_{j_0}(i)=\{\omega\in \Omega \mid \exists \omega'\in E_{j_0} \text{ such that } \omega_k = \omega'_k \text{ for }1\leq k\leq \zeta_j(\omega'), \\\text{ and }\omega_{\zeta_j(\omega')+1}\in B,\text{ and }\omega_{k+1} = \omega'_k \text{ for } k> \zeta_j(\omega') \}\nonumber.
\end{multline}
It is immediate that $\Prob_H(\widehat{E}_{j_0}(i))=p_0 \,\Prob_H(E_{j_0})>0$ and $\widehat{E}_{j_0}(i)\subset \bigcup_{k<j_0}E_k$. Further, for every $\omega\in E_i$, the letters at positions $\zeta_j(\omega)$ cannot be followed by a letter $b\in B$. This is because if such a letter was at this position, by equation (\ref{eq:gapBound}), $\omega\in E_l$ for some $l<i$, a contradiction to the initial hypothesis.
Since $\omega_{\zeta_j(\omega)+1}\notin B$ we must also have $E_{j_0}(i)\cap
E_{j_0}(i')=\varnothing$ for all $i\neq i'$, as otherwise it would contradict the ordering of
\emph{exceptional} positions.
But then
\[ 1=\Prob_H \left(\bigcup_{i<j_0}E_i\right) \geq \Prob_H\left(\bigcup_{i\in\N}\widehat{E}_{j_0}(i)\right)
=\sum_{i\in\N}\Prob_H(\widehat{E}_{j_0}(i))=\sum_{i\in\N} p_0\, \Prob_H(E_{j_0})=\infty, \]
a contradiction. Therefore  
\[
\Prob_H\Big\{ \omega\in\Omega \;\Big|\; \liminf_{k\to\infty}\sum_{e\in \mathbf{C}_{\omega}^k}h(c_e)\in(0,\infty)  \Big\}=0,
\]
proving our main statement.
\end{proof}

\clearpage
\subsection{Gauge functions for random code-trees with necks}\label{sect:generalcase}

The proof for equicontractive systems is fairly straightforward and its generalisation above relied to a heavy degree on the comparability of the evaluation over level sets with the Hausdorff measure. Allowing different contraction rates, this no longer holds for all $\omega$.
Our proof for the general case will imply that they are still comparable almost surely, but we do
not know this \emph{a priori}.
Instead, we will consider splitting up the tree and assigning values to each node corresponding to the minima achieved below it. We will then adopt the method above to obtain a similar contradiction to the positivity and finiteness assumption.

We start by extending the \emph{almost deterministic} condition to random code-trees with necks.

\begin{defn}
Let $\mathbb{L}$ be a family of IFSs that satisfy the UOSC and Condition~\ref{cond:boundedBelow}. Further, let $\Prob$ be a code-tree measure (see Definition~\ref{defn:codetreemeasure}). We call the random code-tree fractal associated with $\Prob$ \emph{almost deterministic} if
\[
\Prob\Big\{ \tau\in\mathcal{T} \;\Big|\; \sum_{e\in\mathbf{T}_\tau^{N_1(\tau)}}(c_e)^s = 1 \Big\} = 1,
\]
where $s$ is the almost sure Hausdorff dimension of $F_\tau$.
\end{defn}

Since $N_1(\tau)=1$ for all random homogeneous attractors, these definitions coincide. 
Note, as in the equicontractive random homogeneous case, that if the random code-tree fractal is not
almost deterministic, there exist $0<p_0, \epsilon,\gamma<1$ such that 
\[
\Prob\Big\{ \tau\in\mathcal{T} \;\Big|\; \sum_{e\in\mathbf{T}_\tau^{N_1(\tau)}}(c_e)^{s-\epsilon} \leq \gamma \Big\} = p_0,
\]
where $s$ is again the almost sure Hausdorff dimension of $F_\tau$. This allows us to establish a similar drop in the gauge function.
\begin{lma}\label{lma:neckgaps}
Let $\mathbb{L}$ be a family of IFS that satisfy the UOSC and Condition~\ref{cond:boundedBelow}. Let $\Prob$ be a random code-tree measure and assume that $F_\tau$ is not almost deterministic. Further, let $h(t)$ be any gauge function. Then there exist $B\subset \mathcal{T}$ and $t_0\in(0,1)$ such that $\Prob(B)=p_0>0$ and for all $0<t\leq t_0$,
\[
\sup_{\tau\in B}\sum_{e\in\mathbf{T}_\tau^{N_1(\tau)}}h(t\cdot c_e)\leq \gamma\, h(t).
\]
\end{lma}
The proof is almost identical to the first part of the proof of Theorem~\ref{thm:mainequicontractive} and omitted here.
If we were able to prove that 
\[
\Haus^h(F_\tau)\asymp \liminf_{k\to\infty} \sum_{e\in\mathbf{T}_\tau^{N_k(\tau)}}h(c_e),
\]
we could use the same strategy as in Theorem~\ref{thm:mainequicontractive}. However, it seems unlikely
to hold in generality as we have no control over the behaviour in between neck levels. We shall use
a different, yet in same ways similar strategy to prove our main theorem.
\begin{theo}[Main Theorem]\label{thm:mainTheo}
Let $\mathbb{L}$ be a family of IFS that satisfy the UOSC and Condition~\ref{cond:boundedBelow}. Let $\Prob$ be a random code-tree measure and assume that $F_\tau$ is not almost deterministic. Further, let $h(t)$ be any gauge function. Then,
\[
\Prob\{ \tau\in\mathcal{T} \mid \Haus^h (F_{\tau})\in\{0,\infty\} \}=1.
\]
In particular, there does not exist a gauge function that gives positive and finite measure almost
surely.
\end{theo}

Before we give the proof we recall the notion of \emph{sections} and \emph{minimal sections}.
\begin{defn}
Let $\Sigma^*$ be the $\cN$-ary tree. A finite subset $M\subset \Sigma^*$ is called a \emph{section} if every long enough node $v\in\Sigma^*$ has an initial word in $M$. That is, there exists some $l_0$ such that for every $l\geq l_0$ and every $v\in\Sigma_l$ there exists $w\in M$ of length $l_w$ such that $v_i = w_i$ for all $1\leq i \leq l_w$.
A section $M$ is referred to as a \emph{minimal section} if no proper subset of $M$ is a section.
\end{defn}
\begin{proof}[Proof of Theorem~\ref{thm:mainTheo}]
Recall that $\mathcal{O}$ is the set guaranteed by the UOSC. We start by assuming, for a
contradiction, that with positive probability $0<\Haus^h(F_\tau)<\infty$. Since we are again dealing
with a tail event, we can further assume that $\Prob\{0<\Haus^h(F_\tau)<\infty\}=1$ by Kolmogorov's
Zero -- One Law.
Instead of relating the Hausdorff measure to level sets, we will relate them to tree sections.

For any section $M\subset\Sigma^*$,
\[
  F_\tau\subseteq \bigcup_{v\in M} \overline{f_{\mathbf{T}_\tau(v)}(\mathcal{O})}
\]
and so 
\[
  \Haus^h(F_\tau)\leq \sum_{v\in M}
  h(\lvert\overline{f_{\mathbf{T}_\tau(v)}(\mathcal{O})}\rvert)=\sum_{v\in M}
  h(c_{\mathbf{T}_\tau(v)}\cdot\lvert\mathcal{O}\rvert)\leq \kappa_\tau \sum_{v\in M}
  h(c_{\mathbf{T}_\tau(v)}),
\]
for some $\kappa_\tau>0$. In particular this holds for all minimal sections, and we will use these to obtain a lower bound to the Hausdorff measure. 
Let $M_k$ be a minimal section such that every $v\in M_k$ has length at least $k$ and let $\mathbf{M}_k$ be the set of all such minimal sections. Assuming the uniform strong separation condition, \ie $f_{\lambda_1}(\Delta)\cap f_{\lambda_2}(\Delta)\neq\varnothing$ implies $\lambda_1=\lambda_2$, where $\lambda_1,\lambda_2\in \Lambda$ and $\Delta$ is as in Definition~\ref{infinityDef2}, it is easy to show that
\[
  \Haus^h(F_\tau)=\lim_{k\to\infty} \; \inf_{M_k\in\mathbf{M}_k} \sum_{v\in M} h(c_{\mathbf{T}_\tau(v)}\cdot\lvert F_\tau \rvert).
\]
We refer the reader to Furstenberg~\cite{FurstenbergBook14} which provides a proof for $h(t)=t^s$.
The general gauge function case is identical and left to the reader. First, notice that by similar
reasons to Theorem~\ref{thm:mainequicontractive} we can, without loss of generality, assume
\[
  \lim_{k\to\infty} \; \inf_{M_k\in\mathbf{M}_k} \sum_{v\in M} h(c_{\mathbf{T}_\tau(v)}\cdot\lvert F_\tau
  \rvert)\asymp \lim_{k\to\infty} \; \inf_{M_k\in\mathbf{M}_k} \sum_{v\in M} h(c_{\mathbf{T}_\tau(v)}).
\] 
However, we are interested in the case where $\mathbb{L}$ satisfies the UOSC rather than the uniform
strong separation condition. To this end, we can modify our sections to only allow elements that do not
overlap, thus artificially forcing strong separation. That is,
\begin{multline}
\Haus^h(F_\tau)\geq \kappa_\tau \lim_{k\to\infty} \; \inf_{M_k\in\mathbf{M}_k} \Big\{
  \sup_{\widehat{M}\subseteq M_k} \Big\{ \sum_{v\in \widehat{M}} h(c_{\mathbf{T}_\tau(v)}) \;\Big|\;
  \overline{f_{\mathbf{T}_\tau(v)}(\mathcal{O})} \cap \overline{f_{\mathbf{T}_\tau(w)}(\mathcal{O})}=\varnothing, \\\text{
  for distinct }v,w\in \widehat{M}\Big\}\Big\}.\nonumber
\end{multline}
for some $\kappa_\tau$.
But, using Lemma~\ref{upperasslem1}, 
\clearpage
\begin{multline}
\lim_{k\to\infty} \; \inf_{M_k\in\mathbf{M}_k} \Big\{ \sup_{\widehat{M}\subseteq M_k} \Big\{
  \sum_{v\in \widehat{M}} h(c_{\mathbf{T}_\tau(v)}) \;\Big|\;  \overline{f_{\mathbf{T}_\tau(v)}(\mathcal{O})} \cap
  \overline{f_{\mathbf{T}_\tau(w)}(\mathcal{O})}=\varnothing,\\ \text{ for distinct }v,w\in
  \widehat{M}\Big\}\Big\}\nonumber
\end{multline}
\[
  \hspace{8cm}\geq  \left(\frac{c_{\min}}{4}\right)^d   \lim_{k\to\infty} \; \inf_{M_k\in\mathbf{M}_k} \sum_{v\in
M} h(c_{\mathbf{T}_\tau(v)}).\nonumber
\]
We conclude,
\[
  \Haus^h(F_\tau)\in (0,\infty) \qquad\text{if and only if}\qquad  \lim_{k\to\infty} \;
  \inf_{M_k\in\mathbf{M}_k} \sum_{v\in M} h(c_{\mathbf{T}_\tau(v)})\in (0,\infty).
\]
and redefine $\kappa_\tau$ such that
\[
\kappa_\tau \cdot \Haus^h(F_\tau) = \lim_{k\to\infty} \; \inf_{M_k\in\mathbf{M}_k} \sum_{v\in M}
h(c_{\mathbf{T}_\tau(v)}),
\]
setting $\kappa_\tau=1$ for all $\tau$ such that $\Haus^h(F_\tau)\in\{0,\infty\}$.

Instead of tree levels, we consider minimal subsections at neck levels. Let $\mathbf{M}^*=\bigcup M_k$ and set
\[
  S_\tau(t)=\inf_{M_k\in\mathbf{M}^*}\sum_{v\in M} h(t\cdot c_{\mathbf{T}_\tau(v)}).
\]
Clearly, 
\begin{align*}
  \lim_{k\to\infty}\sum_{v\in\Sigma_k}S_{\sigma^v\tau}(\mathbf{T}_\tau(v))&=\lim_{k\to\infty}\sum_{v\in\Sigma_{N_k(\tau)}}S_{\sigma^v\tau}(\mathbf{T}_\tau(v))\\
  &=\lim_{k\to\infty}
  \; \inf_{M_k\in\mathbf{M}_k} \sum_{v\in M} h(c_{\mathbf{T}_\tau(v)})=\kappa_\tau \Haus^h(F_\tau),
\end{align*}
as the leftmost sum is non-decreasing in $k$, for all $\tau\in\mathcal{T}$. So, our initial
assumption is also equivalent to $\mathbf{S}_\tau^k=\sum_{v\in\Sigma_k}S_{\sigma^v\tau}(\mathbf{T}_\tau(v))$ converging to some non-zero and finite value from below for almost every $\tau$. 

Let $\mathcal{T}'$ denote the full measure set for which there exist infinitely many necks. We
define
\[E=\{\tau\in\mathcal{T}' \mid \mathbf{S}_\tau^k< \kappa_\tau \Haus^h(F_\tau)\text{ for all
}k\}\] to be the set of realisations that converge properly to its limit. First, we assume that
$\Prob(E)>0$ to derive a contradiction in a similar spirit to the homogeneous case.
Let $k_0$ be such that $c_{\mathbf{T}_\tau(v)}<t_0$ for all $v\in\bigcup_{k\geq k_0}\Sigma_k$. For $i\in\Z$, we analogously define level sets $E_i$ by
\[
E_i := \Big\{ \tau\in\mathcal{T}' \;\Big|\; \lim_{k\to\infty}\mathbf{S}_\tau^k\in(\gamma^{i+1},\gamma^i\,]  \Big\},
\]
where $\gamma$ is given in Lemma~\ref{lma:neckgaps}.
Again, $E^*=\bigcup_{i\in\Z} E_i$ is a disjoint union with full measure and so $\Prob(E^*\cap E)>0$.
Thus there exists $j_0\in\Z$ such that $\Prob(E\cap E_{j_0})>0$. 

For $\tau\in E^*$ we write $\zeta_k(\tau)$ for the $k$-th neck level not less than $k_0$ such that a `jump' in value occurs at the next neck. \ie
\[
\zeta_1(\tau)=\min\Big\{k>k_0 \;\Big|\;  \mathbf{S}_\tau^{N_k(\tau)}<\mathbf{S}_\tau^{N_{k+1}(\tau)}  \Big\}
\]
and
\[
\zeta_j(\tau)=\min\Big\{k> \zeta_{j-1}(\omega) \;\Big|\; \mathbf{S}_\tau^{N_k(\tau)}<\mathbf{S}_\tau^{N_{k+1}(\tau)}\Big\}
\]
First note that for a given $\tau\in E^*$, there must be some $v\in\Sigma^*$ of length
$N_k(\tau)\leq \lvert v\rvert<N_{k+1}(\tau)$ such that
$S_{\sigma^v\tau}(\mathbf{T}_\tau(v))=h(\mathbf{T}_\tau(v))$ as
otherwise the value would not jump. Further, this jump implies that the shift
$\sigma^{v_k}\tau\notin B$ for $v_k=(1,1,\dots,1)\in\Sigma_{N_k}$. Thus we can form a new set
by inserting a new neck block at this jump value that decreases the Hausdorff measure and set
\begin{multline*}
\widehat{E}_{j_0}(i) = \Big\{ \tau\in\mathcal{T}' \mid  \exists b\in B,\exists \tau'\in E_{j_0} \text{ such that } \tau(v)=\tau'(v)\text{ for all }v\in\bigcup_{k=1}^{N_{\zeta_i}(\tau')-1}\Sigma_k,\\*
\text{ and }\sigma^{v}\tau=b\text{ for all }v\in\Sigma_{N_{\zeta_i}(\tau')}\text{ such that }\mathbf{T}(v)\neq\emptyset,\text{ and }\sigma^v\tau=\sigma^{v_{k+1}}\tau'\\*
\text{ for all }v\in\Sigma_{N_{\zeta_i}(\tau')+N_{1}(b)}\text{ such that }\mathbf{T}(v)\neq\emptyset\Big\}.
\end{multline*}
Now $\Prob$ is invariant with respect to neck-shifts and in-between necks are
independent. We can therefore conclude that $\Prob(\widehat{E}_{j_0}(i))=p_0 \Prob(E_{j_0})>0$ and
obtain
\[ 1=\Prob \left(\bigcup_{i<j_0}E_i\right) \geq \Prob\left(\bigcup_{i\in\N}\widehat{E}_{j_0}(i)\right)
=\sum_{i\in\N}\Prob(\widehat{E}_{j_0}(i))=\sum_{i\in\N} p_0\, \Prob(E_{j_0})=\infty, \]
and so $\Prob(E)=0$. 

Since we have reached a contradiction, we must conclude that 
\[
K=\{\tau\in\mathcal{T}' \mid \exists k_1(\tau)>k_0 \text{ such that }\mathbf{S}_\tau^k=\kappa_\tau\Haus^h (F_\tau)\text{ for all }k\geq k_1(\tau)\}
\]
has full measure. We immediately conclude that sums of any section with word length at least
$k_1(\tau)$ give an upper bound.
Since neck levels are independent we can also conclude that $K\cap B\cap\mathcal{T}'$ has full
measure.

Let $S$ be all minimal sections such that all $v\in S$ have length greater than $k_1(\tau)$ and
\[
  \kappa_\tau\Haus^h(F_\tau)\leq\sum_{v\in S}h(\mathbf{T}_\tau(v))<\frac{\kappa_\tau}{\gamma}\Haus^h(F_\tau).
\]
There must exist $v_0\in S$ such that $S_{\sigma^{v_0}}(\mathbf{T}_\tau(v_0))\leq h(\mathbf{T}_\tau(v_0)) <
\gamma^{-1}S_{\sigma^{v_0}}(\mathbf{T}_\tau(v_0))$. Let $l_0$ be the length of $v_0$. If $l_0=N_k(\tau)$ for
some $k$, then the next neck cannot lie in $B$. Similarly, if $N_k(\tau)<l_0\leq N_{k+1}(\tau)$ we
can use the maximal contraction rate $c_{\max}$ and the maximal splitting $\cN$ to conclude that
there cannot be more than $n_0=(N_{k+1}-N_k)\log(\cN (c_{\max})^{s-\epsilon})/\log\gamma$ consecutive necks in $B$. However, under our assumptions, the value of
$S_{\sigma^{v_0}}(\mathbf{T}_\tau(v_0))$ does not depend on the order of the letters past $l_0$ and there cannot be more than $n_0$ occurrences of necks in $B$ after $k_1(\tau)$. Hence there are only finitely many necks shifts in $B$, a contradiction. 
We conclude that, almost surely, the $h$-Hausdorff measure is zero or infinite.
\end{proof}

\clearpage
\section{Exact packing measure}\label{sect:packingquestion}
The packing measure can be considered the dual of the Hausdorff measure. For arbitrary gauge
functions we define it thus.
\begin{defn}
 Let $F\subseteq\R^d$ and $h(t)$ be a gauge function. Define
 \begin{multline*}
   \mathscr{P}_\delta^h (F)=\sup\Big\{ \sum_i h(\lvert B_i\rvert) \mid \left\{ B_i \right\}
   \text{is a countable collection of disjoint balls }\\ \text{centred in $F$ with radii }
 r_i\leq\delta \Big\}
 \end{multline*}
 and set $\mathscr{P}_0^h(F)=\lim_{\delta\to\infty}\mathscr{P}_\delta^h (F)$. The \emph{packing
 measure} is
 \begin{equation}
   \mathscr{P}^h (F) = \inf\left\{ \sum_{i=1}^\infty \mathscr{P}_0^s (F_i) \mid \text{ where
   }f\subseteq \bigcup_{i=1}^\infty F_i \right\}.\label{eq:packingMeasure}
 \end{equation}
\end{defn} 
Which can easily be seen to be similar to the definition of the Hausdorff measure with one important
difference; we need to to take the second infimum (\ref{eq:packingMeasure}) to guarantee the measure
is countably stable.

We note that there are topological conditions that can help us avoid taking the second
infimum. Recall that $\dim_P(F)=\overline \dim_B (F)$ if $F$ is compact and $\overline \dim_B F\cap O =
\overline \dim_B F$ for every open set $O$ that intersects $F$ non-trivially,
see~\cite[Corollary 3.10]{FractalGeo3}. Similarly, we can prove the following Lemma.

\begin{lma}\label{lma:easierpacking}\index{packing measure}
Let $\mathbb{L}$ be a family of IFS and let $\Prob$ be a random code-tree measure with associated
attractor $F_\tau$.
Let $h(t)$ be a doubling gauge function and assume that all maps
$f_\lambda^i\in\mathbb{I}_\lambda\in\mathbb{L}$ are strict contractions such that there
exist $0<c_{\min}\leq c_{\max}<1$ such that 
\[c_{\min} \lvert x-y\rvert \leq\lvert f_\lambda^i (x)-f_\lambda^i(y)
\rvert \leq c_{\max} \lvert x-y\rvert
\]
for all $\lambda\in\Lambda$ and $i$ and all $x,y\in\R^d$. Let $\tau\in\mathcal{T}$,
then 
\begin{equation}
  \mathscr{P}^h_0(F_\tau)=\infty \quad\implies\quad \mathscr{P}^h(F_\tau)=\infty \label{eq:lowerBound}
\end{equation}
and
\begin{equation}
  \mathscr{P}^h_0(F_\tau)=0 \quad\implies\quad \mathscr{P}^h(F_\tau)=0.\label{eq:upperBound}
\end{equation}
\end{lma}
Note that we did not make any assumption on the contractions and separation conditions in this
Lemma.
\begin{proof}
  Equation (\ref{eq:upperBound}) follows from the definition of $\mathscr{P}^h$ and it remains to
  show (\ref{eq:lowerBound}), \ie we need to show that
\[
\inf\left\{ \sum_{i=1}^\infty \mathscr{P}^h_0(E_i)   \mid  F_\tau\subseteq\bigcup_{i=1}^\infty E_i \right\} = \infty
\]
if $\mathscr{P}_0^h(F_\tau)=\infty$. 
Now $F_\tau$ is compact, and so we can assume that $\{E_i\}$ is finite. Thus there
exists $k$ and $j$ such that there exists $v_j\in\Sigma_{N_k(\tau)}$ with
$f_{\mathbf{T}_\tau(v_j)}(F_{\sigma^{v_j}\tau})\subset E_j$. So, for some $n$ dependent on the cover,
\begin{align*}
\mathscr{P}^h(F_\tau)&=\inf\left\{ \sum_{i=1}^n \mathscr{P}^h_0(E_i)   \;\Big|\; F_\tau\subseteq\bigcup_{i=1}^n E_i \right\}\\
&\geq \inf \left\{\mathscr{P}^h_0 (E_j)    \;\Big|\; F_\tau\subseteq\bigcup_{i=1}^n E_i \right\}&& \text{($j$ as above)} \\
&\geq \inf \left\{\mathscr{P}^h_0(f_{\mathbf{T}_\tau(v_j)}( F_{\sigma^{v_j} \tau}))   \;\Big|\; F_\tau\subseteq\bigcup_{i=1}^n E_i \right\}\\
&\geq \inf \left\{\lim_{\delta\to 0} \kappa \mathscr{P}_\delta^h(F_{\sigma^{v_j}\tau})
\;\Big|\; F_\tau\subseteq\bigcup_{i=1}^n E_i \right\}=\infty,&&\text{(a.s.)}
\end{align*}
where the infimum is taken over all finite covers and $\kappa$ is a finite constant arising from the
maximal distortion of the map $f(.)$ (bounded by $c_{\min}^{N_k}$ and $c_{\max}^{N_k}$) and the
doubling of $h$.
\end{proof}

\subsection{Bounds for equicontractive RIFS}
Inspired by the recent progress on the packing measure of random recursive attractors mentioned
above, we would hope that using the gauge $h_1(t,\beta,\gamma)$ should give similar similar
convergence and divergence, depending on the sign of $\gamma$. This can be achieved by considering
the natural dual to $h_1$. Let $s\geq0$, $\gamma\in\R$ and $\beta>0$, we set
\[
h_1^*(t,\beta,\gamma)=t^s \exp\left(-\sqrt{2\beta\log(1/t)\log\log(\beta\log(1/t))}\right)^{1-\gamma}.
\]
We remark that, in light of Lemma~\ref{lma:easierpacking}, we only sketch proofs.
\begin{theo}\label{thm:packinggauge}\index{packing measure}\index{uniform open set condition (UOSC)}
Let $F_{\omega}$ be the random homogeneous attractor associated to the self-similar\index{self-similar set} RIFS $(\mathbb{L},\mu)$ satisfying the UOSC  and suppose that $c_{\lambda}^{i}=c_{\lambda}\in[c_{\min},c_{\max}]$ for every $i\in\{1,\dots,\#\mathbb{I}_\lambda\}$ and $\lambda\in\Lambda$, where $0<c_{\min} \leq c_{\max}<1$.
Let $\epsilon>0$, $s=\ess\dim_H F_\omega=\ess\dim_P F_\omega$ and
$\beta_0^*=\eta_0\Var(\log\fS_{\omega_{1}}^{s})$ for some $\eta_0^*\in\R$ (arising in the proof).
Then $\mathscr{P}^{h^*_{1}(t,\beta^*_0,\epsilon)}(F_{\omega})=\infty$
almost surely.
\end{theo}

\begin{proof}
By Lemma~\ref{lma:easierpacking} we only have to analyse $\lim_{\delta\to0}\mathscr{P}_{\delta}^{h_1^*(t,\beta_0^*,\epsilon)}(F_\omega)$. 
Let $\langle X\rangle$ denote the compact convex hull of $X$. Since $c_\lambda$ is uniformly bounded
away from $0$ and $1$ and $\sup_{\lambda\in\Lambda}\#\mathbb{I}_\lambda<\infty$ there exist $l$
and there exists at least one $e_{ch}(\omega)\in\mathbf{C}_\omega^l$ for which we have
$f_{e_{ch}(\omega)}(\langle F_\omega \rangle)\subset \langle F_\omega \rangle$. Thus we get, in a similar fashion to the Hausdorff measure argument,
\begin{align*}
\lim_{\delta\to0}\mathscr{P}_{\delta}^{h_1^*(t,\beta_0^*,\epsilon)}(F_\omega)
&=\lim_{\delta\to0} \sup\Bigg\{ \sum_{i=1}^\infty h_1^*(2r_i,\beta_0^*, \epsilon)  \;\Big|\; \{B(x_i,r_i)\} \text{ is a disjoint }\\
&\hspace{2cm}\text{collection of balls with $2r_i<\delta$ and $x_i\in F_\omega$}  \Bigg\}\\ &\\
&\geq \limsup_{k\to\infty} \sum_{e\in\mathbf{C}_{\omega}^k} h_1^*(\lvert
f_{e\, e_{ch}(\sigma^{k} \omega)}(F_{\sigma^{k+l} \omega})\rvert,\beta_0^*, \epsilon)\\ 
&\geq \limsup_{k\to\infty} \left(\prod_{i=1}^k \cN_{\omega_i}\right)  h_1^*(c_{\omega_1}c_{\omega_2}\dots c_{\omega_k}c_{\min}^l,\beta_0^*,\epsilon)\\
&\geq  \limsup_{k\to\infty} \left(\prod_{i=1}^k \cN_{\omega_i}\right) \kappa
(c_{\omega_1}c_{\omega_2}\dots c_{\omega_k})^s \exp\bigg(-(1-\epsilon)\\
&\hspace{1cm}\cdot\sqrt{\beta_0^*\log(1/(c_{\omega_1}\dots c_{\omega_k})\log\log(\beta_0^*\log(1/(c_{\omega_1}\dots c_{\omega_k}))}\bigg)\\
&\geq \limsup_{k\to{\infty}}\kappa\exp\left( \sum_{i=1}^k \log\fS_{\omega_i}^s  -  (1-\epsilon)\sqrt{v k \log\log vk}\right)\\
&=\infty,
\end{align*}
writing $v=\Var(\fS_{\omega_1}^s)$ and having used the law of the iterated logarithm\index{law of the iterated logarithm} in the last step.
\end{proof}

Finally, we also obtain an upper bound.

\begin{theo}\label{thm:packinggaugeupper}\index{packing measure}
Let $F_{\omega}$ be the random homogeneous attractor associated to the self-similar RIFS $(\mathbb{L},\mu)$ satisfying the UOSC  and suppose that $c_{\lambda}^{i}=c_{\lambda}\in[c_{\min},c_{\max}]$ for every $i\in\{1,\dots,\#\mathbb{I}_\lambda\}$ and $\lambda\in\Lambda$, where $0<c_{\min} \leq c_{\max}<1$.
Let $\epsilon>0$, $s=\ess\dim_H F_\omega=\ess\dim_P F_\omega$ and
$\beta^*=\eta\Var(\log\fS_{\omega_{1}}^{s})$ for some $\eta^*\in\R$ (arising in the proof), then \[\mathscr{P}^{h^*_{1}(t,\beta^*,\epsilon)}(F_{\omega})=0\]
holds almost surely.
\end{theo}
\begin{proof}
By the homogeneity of the construction
\begin{multline*}
\sup\Bigg\{ \sum_{i=1}^\infty h_1^*(2r_i,\beta_0^*, \epsilon)  \;\Big|\; \{B(x_i,r_i)\} \text{ are disjoint balls with $2r_i<\delta$ and $x_i\in F_\omega$}  \Bigg\}\\
\leq\kappa\sup_{n\geq k(\delta)}\Bigg\{\left(\prod_{i=1}^{n}\cN_{\omega_i}\right) h_1^*(c_{\omega_1}\dots c_{\omega_{n}},\beta^*, \epsilon) \Bigg\}
\end{multline*}
for some $\kappa>0$ depending on the diameter of $F_\omega$ and the doubling properties of $h_1$ only. So, for an appropriately chosen $\eta$, we obtain the desired conclusion from the law of the iterated logarithm.
\end{proof}

\subsection{Existence of a gauge function}
Lemma \ref{lma:easierpacking} is unfortunately not sufficient to allow us to prove the non-existence
of a gauge function with positive and finite packing measure using the same approach
as in Section~\ref{sect:nonExistenceMain}. However, the underlying idea still holds as the packing
measure should, intuitively behave like \[\limsup_{k\to\infty}\sum_{e\in\mathbf{T}_\tau^k}h(c_e).\]
We therefore conjecture
\begin{conj}
  Let $\mathbb{L}$ be a family of IFS that satisfy the UOSC and Condition~\ref{cond:boundedBelow}.
  Let $\Prob$ be a random code-tree measure and assume that $F_\tau$ is not almost deterministic.
  Further, let $h(t)$ be any gauge function. Then,
  \begin{equation*}
    \Prob\left\{ \tau\in\mathcal{T} \mid \mathscr{P}^h(F_\tau)\in\left\{ 0,\infty \right\}
  \right\}=1.
  \end{equation*}
  In particular, there does not exist a gauge function that gives positive and finite measure almost
  surely.
\end{conj}

\section{Implications for a random implicit theorem}\label{sect:implications}
There are two notable implications that our result has for random attractors in general. One
concerns a random analogue of the \emph{implicit theorem} due to Falconer~\cite{Falconer89}, whereas
the other concerns the question on whether $V$-variable models interpolate between random
homogeneous and random recursive sets.

\subsection{The implicit theorems}
The implicit theorems are two statements about metric spaces that give a checkable condition for the
set to have equal Hausdorff and upper-box counting dimension. Further, they give sufficient
conditions for positive and finite Hausdorff measure. Both appeared first in
Falconer~\cite{Falconer89} but can also be found as \cite[Theorems 3.1 and 3.2]{TecFracGeo}.

\begin{prop}
  Let $F$ be a non-empty subset of $\R^d$ and let $a>0$ and $r_0>0$. Write $s=\dim_H F$ and suppose
  that for every set $U$ that intersects $F$ such that $\lvert U\rvert<r_0$ there is a mapping $g :
  U\cap F\to F$ with
  \begin{equation*}
    a\lvert x-y\rvert\leq \lvert U\rvert \cdot \lvert g(x)-g(y)\rvert
  \end{equation*}
  for every $x,y\in F$. Then, $\Haus^s(F)\geq a^s>0$ and the upper box-counting dimension of $F$
  coincides with $s$.
\end{prop}
Heuristically, this means that if every small enough piece of a set $F$ can be embedded into the
entire set $F$ without `too much distortion', the Hausdorff measure is positive and all the
commonly considered dimensions such as Hausdorff, packing, and box-counting dimension, coincide.
Similarly, the second implicit theorem is.
\begin{prop}
  Let $F$ be a non-empty compact subset of $\R^d$ and let $a>0$ and $r_0>0$. Write $s=\dim_H F$ and
  suppose that for every closed ball $B$ with centre in $F$ and radius $r< r_0$ there exists a map
  $g:F\to B\cap F$ satisfying
  \begin{equation*}
    a r \lvert x-y\rvert \leq \lvert g(x)-g(y0\rvert
  \end{equation*}
  for all $x,y\in F$. Then $\Haus^s(F)\leq 4^s a^{-s}<\infty$ and the upper box-counting dimension
  is $s$.
\end{prop}
Here, the intuitive picture is that every ball centred in $F$ contains a `not too small' copy of
the entire set $F$.
We remark that the second theorem can be applied to all self-similar and self-conformal sets and
thus we conclude that their box-counting and Hausdorff dimensions coincide regardless of any overlap
conditions. Further, we can conclude that their Hausdorff measure is always finite.
If we additionally have overlap conditions such as the open set condition, we can apply the first
implicit theorem and get not just finite but also positive measure.

It has been a long-standing question whether there exists some random analogue of such statement.
It is certainly feasible that such a statement can exist, as it is known that the Hausdorff and
box-counting dimensions agree for many common random models such as random recursive, $V$-variable,
and also graph directed models, see \cite{Troscheit17} and references therein.

However, a general statement for random sets that includes a conclusion of positive and finite measure has been more
elusive. It was known for a while that the Hausdorff measure of random recursive sets is $0$ almost
surely for reasonable random self-similar sets, as we have discussed in Section~\ref{sect:survey}.
Any potential implicit theorems with results on the positivity of the Hausdorff measure must have
taken into account the underlying process and would have been associated with a gauge function for
that process.

Our results show, however, that even though random homogeneous sets are very natural and should
surely have come under the scope of such a theorem, the non-existence of a gauge functions means
that there could not be such a general implicit statement.

The best one could hope for for such an implicit theorem is just a statement about the coincidence
of Hausdorff and box-counting dimension, \ie that it does not matter whether one takes the infimum
over all coverings, but restricts oneself to coverings with sets of equal diameter.

\subsection{$V$-variable interpolation}

The $V$-variable model was first introduced to interpolate between the random homogeneous and the
random recursive process. It was suggested in Barnsley et al.~\cite{Barnsley12} that the Hausdorff
dimension of (reasonably picked) $V$-variable sets should interpolate between the two models. That
is, let $F_V$ be the random set created by a $V$-variable process sharing the same RIFS
$(\mathbb{L},\mu)$. Further, denote by $F_\infty$ the random attractor of the associated random
recursive set. Barnsley et al.\ claim in \cite{Barnsley12} that $\ess\dim_H F_V \to \ess\dim_H F_\infty$ as $V\to\infty$ but only
support this with some computational evidence. As far as we are aware, there is no known
proof that the dimension converges. What is more, it is not even known whether this sequence of
dimensions is increasing. The computational evidence seems to suggest the following conjecture.
\begin{conj}
  Let $(\mathbb{L},\mu)$ be a RIFS that satisfies the UOSC and Condition~\ref{cond:boundedBelow}.
  Let $F_V$ be the associated $V$-variable, and $F_\infty$ be the random recursive attractor. Let
  $D : \N \cup\{\infty\} \to \R^+$ be given by $D(V)=\ess\dim_H F_V$ and suppose that
  $D(1)<D(\infty)$. Then $D(V)$ is strictly increasing  and $D(V)\to D(\infty)$ as $V\to\infty$.
\end{conj}

One could further ask whether there is a closed form expression for the Hausdorff dimension of
equicontractive RIFS $V$-variable sets, as there is for random homogeneous and random recursive:
$\E(\log\cN_\lambda)/\log c$ and $\log\E(\cN_\lambda)/\log c$, respectively.

However, our work has shown that $V$-variable sets have much more in common with random homogeneous
processes than with the random recursive. In both of the former there simply cannot be a gauge
function that adequately describes the fine dimension, whereas there is one for the latter.
This implies that $V$-variable processes cannot interpolate the fine dimension (as there is nothing
to interpolate with), but they could still interpolate in the coarse sense.

\subsection*{Acknowledgements}
The work was started while the author attended the ICERM semester
programme \emph{Dimension Theory and Dynamics} in the spring of 2016 and the author thanks ICERM and
Brown University for their financial support.  Further, ST was at initially at the University of St
Andrews and supported by EPSRC Doctoral Training Grant EP/K503162/1 and later by the Faculty of
Mathematics at the University of Waterloo and NSERC grants 2016-03719 and RGPIN-2014-03154.
The author is indebted to Kenneth Falconer, Mike Todd, Jonathan Fraser, and Julius Jonu\v{s}as for many
fruitful discussions. 

\bibliographystyle{../../Biblio/stcustom} \bibliography{../../Biblio/Biblio}

\end{document}